\documentclass[12pt]{elsarticle}

\usepackage[nottoc,notlot,notlof]{tocbibind}

\usepackage{amsfonts,amssymb,amsmath}
\usepackage{lineno,hyperref}

\usepackage{indentfirst}

\usepackage{color}

\usepackage{mathrsfs}

\newtheorem{theorem}{Theorem}[section]
\newtheorem{proposition}{Proposition}[section]
\newtheorem{corollary}{Corollary}[section]
\newtheorem{lemma}[theorem]{Lemma}

\newtheorem{definition}[theorem]{Definition}
\newtheorem{example}[theorem]{Example}
\newtheorem{remark}[theorem]{Remark}

\makeatletter
\def\ps@pprintTitle{%
   \let\@oddhead\@empty
   \let\@evenhead\@empty
   \let\@oddfoot\@empty
   \let\@evenfoot\@oddfoot
}
\makeatother

\newenvironment{proof}[1][\noindent \textbf{Proof: }]{#1}{ \hfill $\square$ \vspace{2mm}}

\begin{document}

\begin{frontmatter}

\title{Globally hypoelliptic triangularizable systems of periodic pseudo-differential operators}

\author{Fernando de \'{A}vila Silva}
\ead{fernando.avila@ufpr.br}

\address{Departamento de Matem\'{a}tica, Universidade Federal do Paran\'{a}, \\ Caixa Postal 19081, Curitiba, PR 81531-990, Brazil}

\begin{abstract}
This article presents  an investigation on the global hypoellipticity problem for systems  belonging to the class $P = D_t + Q(t,D_x)$, where $Q(t,D_x)$ is a $m\times m$ matrix with entries $c_{j,k}(t)Q_{j,k}(D_x)$. The coefficients  $c_{j,k}(t)$ are smooth, complex-valued functions on the torus $\mathbb{T} \simeq 	\mathbb{R}/2\pi\mathbb{Z}$  and $Q_{j,k}(D_x)$ are pseudo-differential operators on $	\mathbb{T}^n$. The approach consists in establishing conditions on the matrix symbol $Q(t,\xi)$ such that it can be transformed into a suitable triangular form $\Lambda(t,\xi) + \mathcal{N}(t,\xi)$, where $\Lambda(t,\xi)$ is the diagonal matrix  $diag(\lambda_{1}(t,\xi) \ldots \lambda_{m}(t,\xi))$ and $\mathcal{N}(t,\xi)$ is a nilpotent upper triangular matrix. Hence,  the global hypoellipticity of $P$ is studied  by analyzing  the behavior of the eigenvalues $\lambda_{j}(t,\xi)$ and its averages $\lambda_{0,j}(\xi)$, as $|\xi| \to \infty$. 
\end{abstract}


\end{frontmatter}

%\linenumbers

\tableofcontents

%%%%%%%%%%%%%%%%%%%%%%%%%%%%%%%%%%%%%%%%%%%%%%%%%%%%%%%%%%%%%%%%
%%%%%%%%%%%%%%%%%%%%%%%%%%%%%%%%%%%%%%%%%%%%%%%%%%%%%%%%%%%%%%%%
\section{Introduction}
%%%%%%%%%%%%%%%%%%%%%%%%%%%%%%%%%%%%%%%%%%%%%%%%%%%%%%%%%%%%%%%%
%%%%%%%%%%%%%%%%%%%%%%%%%%%%%%%%%%%%%%%%%%%%%%%%%%%%%%%%%%%%%%%%

This article discusses the global hypoellipticity problem for systems belonging to the class
\begin{equation}\label{syst-intro}
	P = D_t + Q(t,D_x), \ t\in 	\mathbb{T}, x \in 	\mathbb{T}^n,
\end{equation}
where $D_t = i^{-1} \partial_t$, $\mathbb{T} \simeq \mathbb{R}/2\pi\mathbb{Z}$ stands as the torus,  $Q(t,D_x)$  is a $m\times m$ matrix operator with entries $c_{j,k}(t)Q_{j,k}(D_x)$, where $c_{j,k}(t)$ are smooth complex-valued functions on $\mathbb{T}$ and $Q_{j,k}(D_x)$ are  pseudo-differential operators on $\mathbb{T}^n$.

Let us recall that system $P$  is \textit{globally hypoelliptic} if 
$m$-dimensional vectors $u$ and $f$, with coordinates in $\mathcal{D}'(\mathbb{T}^{n+1})$ and $C^{\infty}(\mathbb{T}^{n+1})$ respectively, satisfy equation $Pu=f$ then the coordinates of $u$ are also smooth functions  on $\mathbb{T}^{n+1}$.

The study of global properties for linear operators    is a challenging problem even for the case of vector fields and differential operators on the torus, see the impressive list  \cite{BCP04,BMZ,BK,BKWS,BKNZ15,BPZaZug17,HIMONASGER,HOU79,Petr11} and the references therein. Investigations on  pseudo-differential  classes are also considered in the literature, for instance,  \cite{C-CHINI,AvilaMedeira,AGKM,DGY02,BCCJ16}. These problems are also analyzed in the setting of compact manifolds and Lie groups, e.g. \cite{Araujo2019,BCM,BPZZ,HZ,KMR}. In particular, we point out that an important tool, present in all these references and here, is a  Fourier analysis characterizing the functional spaces under investigation.

On the other hand, although the analyses of regularity of solutions  (as well as solvability and well-posedness of systems of differential and pseudo-differential operators, defined on $\mathbb{R}^n$)  compose a widely explored problem in the literature, there are apparently  no published works dealing with the global hypoellipticity problem for $m$-dimensional \textit{periodic} systems of pseudo-differential operators of type \eqref{syst-intro}.

Hence, this article presents an approach for investigations of the global hypoellipticity of unexplored classes of systems  on the torus. For an outline of the main results and techniques, let us  consider $P$, as in \eqref{syst-intro}, and  equation $Pu=f$. By using the partial Fourier series, with respect to $x$, we can conclude that the global hypoellipticity problem is equivalent to  analyzing the solutions of systems 
\begin{equation}\label{general-syst-intro}
	D_t \widehat{u}(t, \xi)  + Q(t,\xi) \widehat{u}(t, \xi) =\widehat{f}(t, \xi), \ t \in\mathbb{T}, \ \xi \in \mathbb{Z}^n,
\end{equation}
where $Q(t,\xi) = [c_{j,k}(t)Q_{j,k}(\xi)]$ is the matrix symbol of $Q(t,D_x)$. More precisely,  it is necessary, and sufficient, to prove that all derivatives of the coordinates of $\widehat{u}(t, \xi)$ converge to zero faster than any polynomial (see Proposition \ref{prop-smooth}).

The main problems in this procedure involve the computations of solutions for \eqref{general-syst-intro}. To handle  this obstacle, inspiration was taken from  T. V. Gramchev and M. Ruzhansky (see \cite{Gramchev2013}) and   C. Garetto, C. J\"{a}h and M. Ruzhansky (see \cite{Garetto2018}). These authors present classes of  systems \eqref{general-syst-intro} (defined on $\mathbb{R}^n$) that can be reduced to a triangular form. Hence, following this inspiration, the present article considers classes of systems such that $Q(t,\xi) = [c_{j,k}(t)Q_{j,k}(\xi)]$ has a \textit{smooth triangularization}
$$
S^{-1}(t, \xi) \, Q(t,\xi) \,  S(t, \xi) =  \Lambda(t,\xi) + \mathcal{N}(t,\xi),
$$
with $\Lambda(t,\xi) = diag(\lambda_{1}(t,\xi), \ldots, \lambda_{m}(t,\xi))$, where  $\mathcal{N}(t,\xi) = [r_{j,k}(t,\xi)]$ is a nilpotent upper triangular  matrix. Moreover, this triangularization 
allows us to replace the study of system \eqref{general-syst-intro} by the triangular form
\begin{equation}\label{triangular-system-intro}
	D_t v(t, \xi)  + [\Lambda(t,\xi) + \mathcal{N}(t, \xi)] v(t, \xi) = g(t, \xi), \ t \in \mathbb{T}, \ \xi \in \mathbb{Z}^n,
\end{equation}
where
$$
v(t, \xi) = S^{-1}(t,\xi) \widehat{u}(t,\xi) \ \textrm{ and } \ g(t, \xi) = S^{-1}(t,\xi) \widehat{f}(t,\xi).
$$

However,  to guarantee  equivalence between the behaviors of solutions of \eqref{general-syst-intro} and \eqref{triangular-system-intro} when $|\xi| \rightarrow \infty$, it is necessary that  $S$, and its inverse $S^{-1}$, have a polynomial growth of type
\begin{equation}\label{smooth-S}
	\sup_{t\in \mathbb{T}}\|\partial_t^{\alpha} S(t, \xi)\|_{\mathbb{C}^{m\times m}} \leq C|\xi|^{\gamma}, \ \textrm{ and } \ 
	\sup_{t\in \mathbb{T}}\|\partial_t^{\alpha} S^{-1}(t, \xi)\|_{\mathbb{C}^{m\times m}} \leq C|\xi|^{\gamma}.
\end{equation}
Furthermore, we assume that 
\begin{equation}\label{reg-B}
	\sup_{t\in \mathbb{T}}\|\partial_t^{\alpha}[ 
	S^{-1}(t, \xi)\cdot D_tS(t, \xi)]\|_{\mathbb{C}^{m\times m}} \leq C|\xi|^{-N},
\end{equation}
for all $N>0$.

It is important to point out that a triangularization satisfying \eqref{smooth-S} and \eqref{reg-B} can be easily obtained  for the  constant coefficients case  
$Q(t,D_x) = [Q_{j,k}(D_x)]_{m\times m}$, since in this configuration  we get
\begin{equation}\label{general-syst-intro-const}
	D_t \widehat{u}(t,\xi)  + [Q_{j,k}(\xi)] \widehat{u}(t,\xi) =\widehat{f}(t,\xi), \ \xi \in \mathbb{Z}^{n}.
\end{equation}
Thus, we can use the following, and well known, Schur's triangularization process (see \cite{Hof}): 
\begin{lemma}[Schur's triangularization]\label{Schur}
	Let $A$ be an $m\times m$ complex matrix. Then, there exists a unitary matrix $S$ such that $S^*AS$ is upper triangular with diagonal elements $r_{j,j} = \lambda_j$.	
\end{lemma}

Given this result we have
$$
S^{-1}(\xi) Q(\xi) S(\xi) = \Lambda(\xi) + \mathcal{N}(\xi), \ \forall \xi \in \mathbb{Z}^n.
$$
Moreover, since  $S(\xi)$ and $S^{-1}(\xi)$ are unitary, for every $\xi$, the estimates \eqref{smooth-S} are automatically satisfied. Since $D_tS(\xi) \equiv 0$, we may replace \eqref{general-syst-intro-const} by its equivalent triangular form
\begin{equation*}
	D_t v(t,\xi)  + (\Lambda(\xi) + \mathcal{N}(\xi)) v(t,\xi) =g(t,\xi), \ \xi \in \mathbb{Z}^{n}.
\end{equation*}

Section \ref{sec-1} presents a complete analysis for the constant coefficient case and starts off with an investigation for systems of type $L = [L_{j,k}(D_y)]$,  where $L_{j,k}(D_y)$ is a pseudo-differential operator  on $\mathbb{T}^N$ . Theorem  \ref{LGH} establishes necessary and sufficient conditions for the global hypoellipticity in view of the eigenvalues of the matrix-symbols $L(\eta)$. These conditions are extended  to the system $P=D_t + Q(D_x)$ by Theorem \ref{Th-general-const}. Examples and applications are presented in Subsections \ref{exe-higher-order} and \ref{exe-sum-commuting}.

Furthermore, Subsection \ref{perturbations} discusses the problem of perturbations, that is,  systems of type
$$
L_\epsilon = L + \epsilon Q, \ \epsilon \in \mathbb{C},
$$
with  $Q = [Q_{j,k}(D_y)]$  and $\epsilon$ belonging to some small ball at the origin. The main result that follows  is Theorem \ref{prop-perturbation}, where the essential hypothesis is that  the eigenvalues  of $L_\epsilon(\eta)$ have an analytic expansion on  $\epsilon$ (see \eqref{eigen-expansion}).

Now, concerning the general system  $P=D_t + Q(t,D_x)$ in \eqref{general-syst-intro},  Section 3 introduces the concept of \textit{strongly triangularizable symbols}, according to Definition \ref{def-stro-tri}.  Broadly speaking, this is a class for which  there exists a smooth triangular  form \eqref{triangular-system-intro} satisfying  conditions \eqref{smooth-S} and \eqref{reg-B} (see Theorem \ref{rediction-theorem}).
The main result that follows  is  Theorem \ref{The-Nece-suff-GH}, where it is shown that, under suitable conditions, $P$ is globally hypoelliptic \textit{if and only if} each constant coefficient operator
$$
\mathscr{L}_{0,k} = D_t + \lambda_{0,k}(D_x),   \ k=1, \ldots, m
$$
is globally hypoelliptic on $\mathbb{T}^{n+1}$, where $\lambda_{0,k}(D_x)$ is formally defined by
$$
\lambda_{0,k}(D_x) w(x) =  \sum_{\xi \in \mathbb{Z}^n}{e^{i x \cdot \eta} \lambda_{0,k}(\xi) \widehat{w}(\xi)}, 
$$
and
$\lambda_{0,k}(\xi) = (2 \pi)^{-1}\int_{0}^{2\pi}\lambda_{k}(t, \xi)dt,$ for $\xi \in \mathbb{Z}^n$.

It is important to emphasize that the hypotheses on Theorem \ref{The-Nece-suff-GH} are independent of the order of the operators $Q_{j,k} = Q_{j,k}(D_x)$. In particular, this means that the results can be applied in the case where $Q_{j,k}$ are differential operators. This is a surprising contrast with the scalar case $m=1$, where, in general, it is not possible to apply the equivalence, suggested by Theorem \ref{The-Nece-suff-GH},
for  classes of differential operators.

Regarding  the introduction  of  strongly triangularizable symbols, it is important to emphasize that the main
inspiration comes from  the  approaches  presented  in \cite{Garetto2018,Gramchev2013}. These works handle with the triangularization problem for  matrices of pseudo-differential operators  on $\mathbb{R}^n$ and, in particular, they have shown how to compute the matrices $S$ and $S^{-1}$. Hence, it is proven, in Section \ref{sec-4} (in Theorems \ref{schur-smooth-trian} and \ref{smooth-t-sym})   that the symbol $Q(t, \xi)$ can be strongly triangularized,  under suitable conditions on its eigenvectors.

\subsection{Notations and preliminary results}\label{sec1.1}

Throughout this paper the variables on $\mathbb{T}$, $\mathbb{T}^n$ and $\mathbb{T}^N$ are denoted by $t, x$ and $y$, respectively, with corresponding dual variables  $\tau \in \mathbb{Z}$, $\xi \in \mathbb{Z}^n$ and $\eta \in \mathbb{Z}^N$. 

As usual,  $\mathcal{D}'(\mathbb{T}^N)$ and $C^{\infty}(\mathbb{T}^N)$ stand as the spaces of distributions and smooth functions on $\mathbb{T}^N$, respectively. Also,  $\mathcal{D}'_{m}(\mathbb{T}^N)$ and ${C}^{\infty}_{m}(\mathbb{T}^N)$  denote the spaces of $m$-dimensional vectors with coordinates in  $\mathcal{D}'(\mathbb{T}^N)$ and $C^{\infty}(\mathbb{T}^N)$, respectively.

The analyses will use  the characterizations of $\mathcal{D}'(\mathbb{T}^N)$ and $C^{\infty}(\mathbb{T}^N)$ by means of the  Fourier coefficients.  For this, let us  recall that if $u \in \mathcal{D}'(\mathbb{T}^{N})$, then  its  Fourier coefficients are defined by
\begin{equation*}
	\widehat{u}(\eta) = (2 \pi)^{-N} <u ,  e^{-i y \eta}> \, ,  \
	\eta \in \mathbb{Z}^N.
\end{equation*}

For  $N= n+1$, the partial Fourier coefficients, with respect to $x \in \mathbb{T}^n$, are given by
\begin{equation*}
	\widehat{u}(t, \xi) = (2 \pi)^{-n} <u(t, \cdot),  e^{-i x\xi}>,  \
	\xi \in \mathbb{Z}^n, \ t \in \mathbb{T}.
\end{equation*}

\begin{proposition}\label{prop-smooth}
	Let $\{a_{\eta}\}_{\eta \in \mathbb{Z}^N}$ be a sequence of complex numbers,   $\{c_{\xi}(t)\}_{\xi \in \mathbb{Z}^n}$ be a sequence of smooth functions on $\mathbb{T}$ and  the formal series
	\begin{equation*}
		v(y) =  \sum_{\eta \in \mathbb{Z}^{N}}{a_{\eta} e^{i  y \cdot \eta}}
		\ \textrm{ and } \
		u(t,x) = \sum_{\xi \in \mathbb{Z}^{n}}{c_\xi(t) e^{i  x \cdot \xi}}.
	\end{equation*}
	
	Thus, we have the following:
	
	\begin{enumerate}
		\item [(a)] the series $v$ converges  in $\mathcal{D}'(\mathbb{T}^{N})$ if and only if there  exists positive constants $M$, $C$ and $R$ such that
		\begin{equation}\label{smooth-coef-full}
			|a_{\eta}| \leq C |\xi|^{M}, \ |\xi|\geq R.
		\end{equation}
		Moreover, $v(y) \in C^{\infty}(\mathbb{T}^{N})$ if and only if estimate \eqref{smooth-coef-full} is fulfilled for every $M<0$. In both cases we have
		$\widehat{v}(\eta) = a_{\eta}$.

		\item [(b)] the series $u$ 
		converges in $\mathcal{D}'(\mathbb{T}^{n+1})$ if and only if, for any $\alpha \in \mathbb{Z}_+$, there are positive constants $M$, $C$ and $R$ such that
		\begin{equation}\label{part-smooth-coef}
			\sup_{t \in \mathbb{T}}|\partial^{\alpha}_t c_{\xi}(t)| \leq C |\xi|^{M}, \ |\xi|\geq R.
		\end{equation}
		Moreover, $u \in C^{\infty}(\mathbb{T}^{n+1})$ if and only if estimate \eqref{part-smooth-coef} holds true for every $M<0$. In both cases we have 
		$c_{\xi}(\cdot) = \widehat{u}(\cdot, \xi)$.
		
	\end{enumerate}
\end{proposition}

By $\Psi^{\nu}(\mathbb{T}^N)$ we denote the class of pseudo-differential operators on $\mathbb{T}^N$, of order $\nu \in \mathbb{R}$ and acting on $C^{\infty}(\mathbb{T}^N)$,  formally defined by 
\begin{equation}\label{pseudo-const}
	a(D_y) u(y) =  \sum_{\eta \in \mathbb{Z}^N}{e^{i y \cdot \eta} a(\eta) \widehat{u}(\eta)},
\end{equation}
with symbol $a(\eta) = \{a(\eta)\}_{\eta \in \mathbb{Z}^N}$ satisfying
\begin{equation}\label{ineq-symbol-const}
	|a(\eta)| \leq C |\eta|^{\nu}, \ \forall \eta \in \mathbb{Z}^{N}.
\end{equation}

For further information  regarding the quantization of pseudo-differential operators on the torus, see  M. Ruzhansky and V. Turunen in  \cite{RT3}.

The next Theorem is an  extension of the results presented in Greenfield's and Wallach's work in \cite{GW1}.

\begin{theorem}\label{GW}
	The operator $a(D_y)$, given by \eqref{pseudo-const}, is globally hypoelliptic on $\mathbb{T}^N$ if and only if there exists positive constants $C$, $M$ and $R$ such that
	\begin{equation*}
		|a(\eta)|  \geq C |\eta|^{-M}, \ |\eta|  \geq R.
	\end{equation*}	
\end{theorem}

In reference \cite{AGKM}, authors R. B. Gonzalez, A.  Kirilov, C. Medeira and F. de {\'A}vila Silva,  
characterize the global hypoellipticity of operators belonging to the class
\begin{equation}\label{P-AGKM}
	P = D_t+c(t)a(D_x),  (t, x) \in \mathbb{T}^{n+1},
\end{equation}
in view of the functions 
\begin{equation*}
	t\in\mathbb{T}\mapsto \lambda(t,\xi)\doteq c(t)a(\xi), \ \xi \in \mathbb{Z}^n,
\end{equation*}
and its averages
\begin{equation*}
	\lambda_{0}(\xi) = (2 \pi)^{-1} \int_{0}^{2 \pi} \lambda(t, \xi) dt, \ \xi \in \mathbb{Z}^n.
\end{equation*}

The following Theorem  summarizes some of the results in \cite{AGKM} that are related with the present  investigations.

\begin{theorem}\label{AGKM} 
	Let $P$ be as in \eqref{P-AGKM}, set $c(t) = p(t) + i q(t)$ and 	$a(\xi) = \alpha(\xi) + i \beta(\xi)$.
	
	\begin{enumerate}
		\item [(a)] if $P$ is globally hypoelliptic, then  the set 
		$$
		Z_{P} = \{\xi \in \mathbb{Z}^n; \, \lambda_{0}(\xi) \in\mathbb{Z} \}
		$$		
		is finite and the operator $\mathscr{P}_{0} = D_t + c_{0}a(D_x)$ is globally hypoelliptic, where $c_{0}$ denotes the average $c_{0} = (2 \pi)^{-1} \int_{0}^{2 \pi} c(t) dt$;

		\item [(b)] if $\mathscr{P}_{0}$ is  globally hypoelliptic  and the functions 
		$$
		\mathbb{T} \ni t \mapsto \Im \lambda(t,\xi) = p(t)\beta(\xi) + q(t)\alpha(\xi) 
		$$
		do not change sign, for sufficiently large $|\xi|$, then $P$ is globally hypoelliptic.
		
		\item [(c)]  the following statements are equivalent:
		\begin{itemize}
			\item[i)] $\mathscr{P}_{0}$  is globally hypoelliptic;
			\item[ii)]  there exist positive constants $C$, $M$ and $R$ such that
			\[|\tau+\lambda_{0}(\xi)|\geq C(|\tau| + |\xi|)^{-M}, \ \textrm{for all} \ |\tau| + |\xi| \geq R;\]
			\item[iii)] there exist positive constants $\widetilde{C}$, $\widetilde{M}$ and $\widetilde{R}$ such that
			\[|1-e^{\pm2\pi i\lambda_{0}(\xi)}|\geq \widetilde{C}|\xi|^{-\widetilde{M}}, \ \textrm{for all} \  |\xi|\geq \widetilde{R}.\]
		\end{itemize}
		
	\end{enumerate} 
	
\end{theorem}

%%%%%%%%%%%%%%%%%%%%%%%%%%%%%%%%%%%%%%%%%%%%%%%%%%%%%%%%%%%%%%%%
%%%%%%%%%%%%%%%%%%%%%%%%%%%%%%%%%%%%%%%%%%%%%%%%%%%%%%%%%%%%%%%%
\section{Systems with constant coefficients} \label{sec-1}
%%%%%%%%%%%%%%%%%%%%%%%%%%%%%%%%%%%%%%%%%%%%%%%%%%%%%%%%%%%%%%%%
%%%%%%%%%%%%%%%%%%%%%%%%%%%%%%%%%%%%%%%%%%%%%%%%%%%%%%%%%%%%%%%%

This section discusses  global hypoellipticity for systems with constant coefficients. The starting  point is an analysis of the case
$L = [L_{j,k}(D_y)]_{m \times m}$, where $L_{j,k}(D_y) \in \Psi^{\nu}(\mathbb{T}^N)$ is a pseudo-differential operator with symbol $L_{j,k}(\eta)$ and order $\nu \in \mathbb{R}$. 

Notice that, if  $u \in \mathcal{D}'_{m}(\mathbb{T}^N)$ is  a solution of $Lu= f \in C_{m}^{\infty}(\mathbb{T}^N)$, then we obtain  the   systems 
$$
L(\eta) \widehat{u} (\eta) = \widehat{f} (\eta), \ \forall \eta \in \mathbb{Z}^N,
$$
where $L(\eta) = [L_{j,k}(\eta)]_{m \times m}$, $\eta \in \mathbb{Z}^N$.

\begin{proposition}\label{finite_set}
	If  system $L$ is globally hypoelliptic, then the set
	$$
	Z_{\det L} = \{\eta \in \mathbb{Z}^N; \ \det L(\eta)=0 \}
	$$
	is finite.
\end{proposition}

\begin{proof}
	Suppose that   $Z_{\det L}$ is infinite and choose sequences $\{\eta_{j}\}_{j \in \mathbb{N}} \subset Z_{\det L}$ and $\{a_j\}_{j \in \mathbb{N}} \subset \mathbb{C}^m$ such that $\|a_j\|_{\mathbb{C}^m} = 1$ and $L(\eta_j)a_j =0$. Hence, 
	$$
	\widehat{u}(\eta) = 
	\left\{
	\begin{array}{l}
		a_j, \ \textrm{ if } \ \eta = \eta_j,\\
		0, \ \textrm{ if } \ \eta \neq \eta_j,
	\end{array}
	\right.
	$$
	defines a vector $u \in \mathcal{D}'_{m}(\mathbb{T}^N) \setminus C_{m}^{\infty}(\mathbb{T}^N)$. Once $Lu=0$, then $L$ is not globally hypoelliptic.
	
\end{proof}

Note that if  $Z_{\det L}$ is finite, then there is some positive constant $R$
such that $L(\eta)$ is invertible for all $|\eta|\geq R$. In this case, $\widehat{u}(\eta)$ is uniquely defined by $\widehat{u}(\eta) =  L(\eta)^{-1} \widehat{f}(\eta)$, and  we can write 
$$
L(\eta)^{-1} = \dfrac{\mbox{ad} L(\eta)}{\det L(\eta)}, \ |\eta|\geq R,
$$
where  $\mbox{ad} L(\eta)$ is the cofactor matrix of $L(\eta)$, namely, the transpose of the matrix with entries 
$$
\textbf{a}_{j,k}(\eta) = (-1)^{j+k}  \det L_{j|k}(\eta),
$$
where $L_{j|k}(\eta)$ is the $(m-1) \times (m-1)$ matrix constructed by deleting  the  $j$th row  and the $k$th column of $L(\eta)$.

Since 	the entries  $\textbf{a}_{j,k}(\eta)$ are bounded by some polynomial, it follows from the  equalities 
$$
\widehat{u}_j(\eta) = \dfrac{\sum_{i=1}^{m}\textbf{a}_{j,i}(\eta)f_i(\eta)}{\det L(\eta)}, \ j=1, \ldots, m,
$$
that a sufficient condition for the global hypoellipticity is obtained by considering a controlled growth on the determinant $\det L(\eta)$, as stated by the next Proposition.

\begin{proposition}\label{prop-suf-det}
	If there exist  positive constants  $C, M$ and $R$ such that 
	\begin{equation*}
		|\det L(\eta)|  \geq C  |\eta|^{-M}, \ |\eta|\geq R,
	\end{equation*}
	then system $L$ is globally hypoelliptic.
\end{proposition}

Our next step is to study the global hypoellipticity of $L$ by analyzing the behavior of the eigenvalues $\lambda_{1}(\eta),  \ldots, \lambda_m(\eta)$ of the symbols $L(\eta)$. For this,  assume, without loss of generality, that 
\begin{equation*}
	|\lambda_1(\eta)| \leq |\lambda_2(\eta)| \leq \ldots \leq |\lambda_m(\eta)|, \ \eta \in \mathbb{Z}^N,
\end{equation*}
and define the sequences $\{\lambda_{j}(\eta)\}_{\eta \in \mathbb{Z}^N}$, for $j =1, \ldots, m$.

\begin{theorem}\label{LGH}
	System $L$ is globally hypoelliptic if and only if there are positive constants $C_j$, $M_j$ and $R_j$ satisfying
	\begin{equation}\label{det-formula-1}
		|\lambda_{j}(\eta)|  \geq C_j  |\eta|^{-M_j}, \ |\eta|\geq R_j,
	\end{equation}
	for $j =1, \ldots, m$.
\end{theorem}

\begin{proof}
	The sufficiency is a consequence of the identity $\det L(\eta) =\prod_{j=1}^{m} \lambda_j(\eta)$ and Proposition \ref{prop-suf-det}. Conversely, assume that
	$\{\lambda_{k}(\eta)\}_{\eta \in \mathbb{Z}^N}$ does not satisfy \eqref{det-formula-1}, for some $k \in \{1, \ldots, m\}$. In this case, there exists
	$\{\eta_\ell\}_{\ell \in \mathbb{N}} \subset \mathbb{Z}^N$  satisfying $|\lambda_{k}(\eta_\ell)| <  |\eta_\ell|^{-\ell}$, for all $\ell \in \mathbb{N}$. Let 
	$\{v_{k}(\eta)\}_{\eta \in \mathbb{Z}^N}$ be
	a sequence of unitary eigenvectors associated  to  
	$\lambda_{k}(\eta)$ and set
	$$
	\widehat{u}(\eta) = 
	\left\{
	\begin{array}{l}
		v_k(\eta_\ell), \ \textrm{ if } \ \eta = \eta_\ell,\\
		0, \ \textrm{ if } \ \eta \neq \eta_\ell.
	\end{array}
	\right.
	\ \textrm{ and } \
	\widehat{f}(\eta) = 
	\left\{
	\begin{array}{l}
		\lambda_{k}(\eta_\ell)v_k(\eta_\ell), \ \textrm{ if } \ \eta = \eta_\ell,\\
		0, \ \textrm{ if } \ \eta \neq \eta_\ell.
	\end{array}
	\right.
	$$
	
	These sequences define vectors 
	$u \in \mathcal{D}'_{m}(\mathbb{T}^N) \setminus C_{m}^{\infty}(\mathbb{T}^N)$ and  $f \in  C_{m}^{\infty}(\mathbb{T}^N)$ such that $Lu=f$, hence  $L$ is not globally hypoelliptic.
	
\end{proof}

\begin{remark}
	If the original system is triangular, that is, 
	$$
	L(D_y) =
	\left[
	\begin{array}{ccccc}
		\lambda_1(D_y) & r_{1,2}(D_y)  & \ldots & r_{1,m}(D_y) \\
		0 & \lambda_2(D_y)     & \ldots & r_{2,m}(D_y) \\
		\vdots     & \vdots  & \vdots & \vdots \\
		0 & 0      & \ldots  & \lambda_m(D_y) \\
	\end{array}
	\right],
	$$
	then the operators $ r_{j,k}(D_y)$ play no role on the global hypoellipticity, since  the eigenvalues of $L(\eta)$ do not depend on  $r_{j,k}(\eta)$.
	
\end{remark}

\subsection{Systems on $\mathbb{T} \times \mathbb{T}^n$}

Let us  apply the previous results  for systems in the class   
\begin{equation*}
	P = D_t + Q(D_x), \ (t, x) \in \mathbb{T} \times \mathbb{T}^n,
\end{equation*}
where $Q(D_x)=[Q_{j,k}(D_x)]_{m \times m}$ and  $Q_{j,k}(D_x) \in \Psi^{\nu}(\mathbb{T}^n)$. We can start  by observing  that  equation $Pu=f$ is equivalent to 
$$
P(\tau, \xi)\widehat{u}(\tau, \xi) = \widehat{f}(\tau, \xi), \ (\tau, \xi) \in \mathbb{Z} \times \mathbb{Z}^n,
$$
where  $P(\tau, \xi) = \tau  + Q(\xi)$, with $\tau = I_{m\times m}\tau$.

By applying Schur's Lemma \ref{Schur} to $Q(\xi)$ we obtain a unitary matrix $S(\xi)$ such that 
\begin{equation*}
	S^*(\xi) Q(\xi) S(\xi) =  \Lambda(\xi) + \mathcal{N}(\xi), \ \xi \in \mathbb{Z}^n,
\end{equation*}
where  $\Lambda(\xi)= diag(\kappa_{1}(\xi),  \ldots, \kappa_m(\xi))$   and 
$$
\mathcal{N}(\xi) =
\left[
\begin{array}{ccccc}
	0 & r_{1,2}(\xi) & r_{1,3}(\xi) & \ldots & r_{1,m}(\xi) \\
	0 & 0      & r_{2,3}(\xi) & \ldots & r_{2,m}(\xi) \\
	\vdots     & \vdots & \vdots & \vdots & \vdots \\
	0 & 0      & \ldots & 0 &  r_{m-1,m}(\xi) \\
	0 & 0      & \ldots & \ldots & 0 \\
\end{array}
\right].
$$

Since the matrices  $S(\xi)$ and  $S^{-1}(\xi)$ are unitary, we obtain
$$
\|Q(\xi)\|_{\mathbb{C}^{m \times m}} = \|\Lambda(\xi) + \mathcal{N}(\xi)\|_{\mathbb{C}^{m \times m}}, \ \forall \xi \in \mathbb{Z}^n,
$$
hence, it follows from \eqref{ineq-symbol-const} the existence of  positive constants $C_1$ and  $C_2$ satisfying
$$
|\kappa_{j}(\xi)| \leq C_1 |\xi|^{\nu} \ \textrm{ and } \ 
|r_{j,k}(\xi)| \leq C_2 |\xi|^{\nu}, \ |\xi| \to \infty.
$$ 

In particular, we can define the operators $\kappa_{j}(D_x)$ and $r_{j,k}(D_x)$, given by the symbols $\{\kappa_{j}(\xi)\}_{\xi \in \mathbb{Z}^n}$ and $\{r_{j,k}(\xi)\}_{\xi \in \mathbb{Z}^n}$, respectively. Also, we have the triangular system
$$
\mathscr{P} = D_t + \Lambda(D_x) + \mathcal{N}(D_x), \ (t, x) \in \mathbb{T}^{n+1},
$$
where $\Lambda(D_x) = diag(\kappa_1(D_x), \ldots, \kappa_m(D_x))$ and $\mathcal{N}(D_x)=[r_{j,k}(D_x)]$. The symbols of $\mathscr{P}$ are
$$
\mathscr{P}(\eta)=\mathscr{P}(\tau, \xi) = 
\left[
\begin{array}{ccccc}
	\lambda_1(\tau, \xi) & r_{1,2}(\xi) & r_{1,3}(\xi) & \ldots & r_{1,m}(\xi) \\
	0 & \lambda_2(\tau, \xi)& r_{2,3}(\xi) & \ldots & r_{2,m}(\xi) \\
	\vdots & \vdots & \vdots & \vdots & \vdots \\
	%0 & 0 & \ldots & \lambda_{k-1}(\eta) &  r_{k-1k} \\
	0 & 0 & \ldots & \ldots & \lambda_m(\tau, \xi) \\
\end{array}
\right], \ \eta = (\tau, \xi) \in \mathbb{Z} \times \mathbb{Z}^n,
$$
with eigenvalues
$\lambda_j(\eta) = \lambda_j(\tau, \xi) = \tau + \kappa_{j}(\xi).$

We claim that  $P$ is globally hypoelliptic if and only if $\mathscr{P}$ is globally hypoelliptic. Indeed, assume that  $P$ is globally hypoelliptic, let
$v \in \mathcal{D}'_{m}(\mathbb{T}^{n+1})$ be a solution of 
$\mathscr{P} v = g \in C_{m}^{\infty}(\mathbb{T}^{n+1})$ and consider  the corresponding systems 
$\mathscr{P}(\tau, \xi)\widehat{v}(\tau, \xi) = \widehat{g}(\tau, \xi)$. 

Now, set 
$u(\tau, \xi) = S(\xi) \widehat{v}(\tau, \xi)$, 
$f(\tau, \xi) = S(\xi) \widehat{v}(\tau, \xi)$,
$$
u(t,x) = \sum_{(\tau, \xi) \in \mathbb{Z}^{n+1}}{u(\tau, \xi) e^{i  (t,x) \cdot (\tau, \xi)}}
\ \textrm{ and } \ 
f(t,x) = \sum_{(\tau, \xi) \in \mathbb{Z}^{n+1}}{f(\tau, \xi) e^{i  (t,x) \cdot (\tau, \xi)}}.
$$
Since $S(\xi)$ is unitary, we have $f \in C_{m}^{\infty}(\mathbb{T}^{n+1})$ and $u \in \mathcal{D}'_{m}(\mathbb{T}^{n+1})$. It follows from equations
$$
P(\tau, \xi)u(\tau, \xi) = f(\tau, \xi), \ (\tau, \xi) \in \mathbb{Z} \times \mathbb{Z}^n,
$$
that $Pu = f$. Hence, $u \in C_{m}^{\infty}(\mathbb{T}^{n+1})$. Moreover, by identity  $\widehat{v}(\tau, \xi) = S^*(\xi) u(\tau, \xi)$,  we get 
$v \in C_{m}^{\infty}(\mathbb{T}^{n+1})$. Then, system $\mathscr{P}$ is globally hypoelliptic. The other direction can be proved with a similar argument.

With these discussions we have proved the following:

\begin{theorem}\label{Th-general-const}
	System $P = D_t + Q(D_x)$  is globally hypoelliptic if and only if there are positive constants $C_j$, $M_j$ and $R_j$ such that
	\begin{equation*}
		|\tau + \kappa_{j}(\xi)|  \geq C_j  (|\tau| + |\xi|)^{-M_j}, \ |\tau| + |\xi|  \geq R_j,
	\end{equation*}
	for $j =1, \ldots, m$.
\end{theorem}

\begin{corollary}\label{coro-diag-syst-const} The following statements hold true:	
	\begin{enumerate}
		\item [(a)] If  $P$ is globally hypoelliptic, then the  
		set 
		$$
		Z_{\det P} = Z_{\det \mathscr{P}} =  \{(\tau, \xi) \in \mathbb{Z}^{n+1}; \ \det \mathscr{P}(\tau, \xi) = 0 \}
		$$
		is finite.
		
		\item [(b)]  $P$ is globally hypoelliptic if and only if the operators $L_j = D_t + \kappa_j(D_x)$ are globally hypoelliptic.
	\end{enumerate}
\end{corollary}

\begin{proof}
	Part a) is a direct consequence of Proposition \ref{finite_set}, while the proof of b) is given by combining Theorem \ref{AGKM}, part (c),  and Theorem \ref{Th-general-const}.
	
\end{proof}

%%%%%%%%%%%%%%%%%%%%%%%%%%%%%%%%%%%%%%%%%%%%%%%%%%%%%%%%%%%%%%%%
%%%%%%%%%%%%%%%%%%%%%%%%%%%%%%%%%%%%%%%%%%%%%%%%%%%%%%%%%%%%%%%%
\subsection{Example: higher order equations}\label{exe-higher-order}
%%%%%%%%%%%%%%%%%%%%%%%%%%%%%%%%%%%%%%%%%%%%%%%%%%%%%%%%%%%%%%%%
%%%%%%%%%%%%%%%%%%%%%%%%%%%%%%%%%%%%%%%%%%%%%%%%%%%%%%%%%%%%%%%%

This example  discusses the hypoellipticity of the operator
\begin{equation*}
	P = D_t^m + \sum_{j=1}^{m}{D_t^{m-j}Q_j(D_x)}, \ (t, x) \in \mathbb{T}^{n+1},
\end{equation*}
where  $Q_j(D_x) \in \Psi^{\nu}(\mathbb{T}^n)$. The equation $Pu=f$ is equivalent to 
\begin{equation*}
	D_t^m\widehat{u}(t,\xi) + \sum_{j=1}^{m}{q_j(\xi)}D_t^{m-j}\widehat{u}(t,\xi) = \widehat{f}(t,\xi), \ \xi \in \mathbb{Z}^n,
\end{equation*}
which can be rewritten as the system
\begin{equation}\label{syst-ex1}
	D_t {v}(t,\xi)
	-
	\left[
	\begin{array}{ccccc}
		0 & 1 & 0 & \ldots & 0 \\
		0 & 0 & 1 & \ldots & 0 \\
		\vdots & \vdots & \vdots & \vdots & \vdots \\
		-q_m(\xi) & -q_{m-1}(\xi) & \ldots & \ldots & -q_1(\xi) \\
	\end{array}
	\right] 
	{v}(t,\xi)
	=
	{g}(t,\xi),
\end{equation}
by the standard transformations
$$
v_1(t,\xi) = \widehat{u}(t,\xi), \ v_k(t,\xi) = D_t^{k-1}\widehat{u}(t,\xi), \ k = 2, \ldots, m
$$
and $g_1 = \ldots = g_{m-1} = 0$, $g_m = \widehat{f}$.

Let $Q(D_x)$ be the matrix-operator with symbols given by matrix in \eqref{syst-ex1} and set $\widetilde{P} = D_t + Q(D_x)$. Clearly,  operator $P$ is globally hypoelliptic if and only if the related system $\widetilde{P}$ is also globally hypoelliptic. 

As an illustration, consider the second order operator 
$$
P = D^{2}_t - 2\alpha D_t (-\Delta_x)^{1/2} + \beta^2 \Delta_x, 
$$
where $\alpha, \beta \in \mathbb{R}$ and $\Delta_x = -\sum_{j=1}^{n}\partial^2_{x_j}$, for $x \in \mathbb{T}^n$.

Then,  we have 
\begin{equation*}
	D_t {v}(t,\xi)
	-
	\left[
	\begin{array}{cc}
		0 & 1   \\
		-\beta^2|\xi|^2 & 2\alpha|\xi| 
	\end{array}
	\right] 
	{v}(t,\xi)
	=
	{g}(t,\xi),
\end{equation*}
with eigenvalues 
$$
\lambda_1(\tau, \xi) = \tau +|\xi|\rho_1
\ \textrm{ and } \
\lambda_2(\tau, \xi) = \tau +|\xi|\rho_2,
$$
where $\rho_{1,2} =  -\alpha \pm \sqrt{\alpha^2 - \beta^2}.$

If 	$\alpha^2 - \beta^2 <0$, then $\rho_{1,2}$  are not real  and $P$ is globally hypoelliptic. However, the case of real roots could be much more complicated since, in general, we can not apply the usual approximations by rational numbers, once $\tau / |\xi|$ may be irrational. For instance, when $\alpha=\beta$, it is necessary to analyze
$$
\left|\dfrac{\tau}{|\xi|}  - \alpha  \right|, \ |\xi| \rightarrow \infty.
$$

A complete discussion  for these  type of approximations can be found in \cite{AGKM}.

%%%%%%%%%%%%%%%%%%%%%%%%%%%%%%%%%%%%%%%%%%%%%%%%%%%%%%%%%%%%%%%%
%%%%%%%%%%%%%%%%%%%%%%%%%%%%%%%%%%%%%%%%%%%%%%%%%%%%%%%%%%%%%%%%
\subsection{Example: sum of commutative  systems}\label{exe-sum-commuting}
%%%%%%%%%%%%%%%%%%%%%%%%%%%%%%%%%%%%%%%%%%%%%%%%%%%%%%%%%%%%%%%%
%%%%%%%%%%%%%%%%%%%%%%%%%%%%%%%%%%%%%%%%%%%%%%%%%%%%%%%%%%%%%%%%

Let  $\{A_j\}_{j=1}^{n}$ be a family of commuting $m\times m$ complex matrices with eigenvalues $\{\kappa_{j,\ell}\}_{\ell=1}^{m}$, for $j=1, \ldots, n$. Consider $Q_j(D_{x_j}) \in \Psi^{\nu_j}(\mathbb{T}_{x_j})$, the system
\begin{equation*}
	P = D_t + \sum_{j=1}^{n}  A_j Q_j(D_{x_j}), \ (t, x_1, \ldots, x_n) \in \mathbb{T} \times \mathbb{T}^n
\end{equation*}
and assume the following:
\begin{lemma}[Simultaneous Schur's triangularization]\label{Simult-Schur}
	Let  $\mathcal{F}$ be a family of commuting  $m\times m$ complex matrices. Then, there exists a unitary matrix $S$ such that $S^*AS$ is triangular, for any $A \in \mathcal{F}$. 	
\end{lemma}

Now, let $u$ be a solution of $Pu=f$ and consider the related system 
$P(\tau, \xi) \widehat{u}(\tau,\xi)=\widehat{f}(\tau,\xi)$, where
$$
P(\tau, \xi) =  \tau  + \sum_{j=1}^{n} q_j(\xi_j) A_j, \ (\tau, \xi) \in \mathbb{Z} \times \mathbb{Z}^n, \ \tau =  I_{m\times m}\tau.
$$

It follows from Lemma \ref{Simult-Schur} the existence of a  unitary matrix $S$ such that 
\begin{equation*}
	S^* A_{j} S = diag(\kappa_{j,1}, \ldots, \kappa_{j,m}) + \mathcal{N}_j, \ j =1, \ldots, n, 
\end{equation*}
hence, we obtain the equivalent triangular system $\mathscr{P}(\tau, \xi)v(\tau, \xi)=g(\tau, \xi)$ given by 
$$
\mathscr{P}(\tau, \xi) =  \tau  + \sum_{j=1}^{n} q_j(\xi_j) (\Lambda_j + \mathcal{N}_j), \ (\tau, \xi) \in \mathbb{Z} \times \mathbb{Z}^n, \ \tau =  I_{m\times m}\tau.
$$

Therefore,  $P$  is globally hypoelliptic if and only if there exist positive constants $C_{\ell}$, $M_{\ell}$ and $R_{\ell}$ such that 
\begin{equation}\label{exe-sum}
	\left|\tau  + \sum_{j=1}^{n} q_j(\xi_j) \kappa_{j,\ell}\right|  \geq C_{\ell} |(\tau,\xi)|^{-M_{\ell}}, \ |\tau|+|\xi| \geq R_{\ell},
\end{equation}	
for  $\ell  =1, \ldots, m$.

As an example, consider the differential case $Q_j(D_{x_j}) = D_{x_j}$ and assume that each $A_j$ is symmetric. Thus,  the conditions in \eqref{exe-sum} are connected with the simultaneous approximations of  $\{\kappa_{1,\ell}, \ldots, \kappa_{m,\ell}\}$ by rational numbers, that is, the analysis of the quantities
$$
|\tau  + \kappa_{1,\ell} \xi_1 + \ldots + \kappa_{n,\ell} \xi_n|, \, \tau \in \mathbb{Z}, \,  \xi = (\xi_1, \ldots, \xi_n) \in \mathbb{Z}^n,
$$
for each $\ell =1, \ldots, m$.

\subsection{Perturbations of globally hypoelliptic systems}\label{perturbations}

This subsection discusses the hypoellipticity of the perturbed system
\begin{equation*}
	L_\epsilon = L + \epsilon Q, \ \epsilon \in \mathbb{C},
\end{equation*}
for  $\epsilon$  in some small ball at the origin, where $Q = [Q_{j,k}(D_y)]$ and $L_0 = L = [L_{j,k}(D_y)]$  stands for the unperturbed system. The eigenvalues of $L(\eta)$ and $Q(\eta)$ are denoted by $\lambda_j(\eta)$ and $\kappa_j(\eta)$, respectively.

The main goal is to investigate the effect of $Q$ on the global hypoellipticity of  $L$, which is equivalent to analyze the eigenvalues  $\lambda_{\epsilon,j}(\eta)$ of  $L_{\epsilon}(\eta)$ in terms of $\lambda_j(\eta)$ and $\kappa_j(\eta)$.

If we admit the commutative hypothesis $[L,Q] = 0$, then
we can apply a simultaneous  triangularization to the matrices $L(\eta)$ and $Q(\eta)$. Furthermore,  system 
$L_{\epsilon}(\eta)\widehat{u}(\eta) = \widehat{f}(\eta)$ is equivalent to  the corresponding triangular form
$(\Lambda_{\epsilon} + \mathcal{N}_{\epsilon}) v(\eta) = g(\eta),$
where $\Lambda_{\epsilon}$ is the diagonal matrix with entries 
$$
\lambda_{\epsilon,j}(\eta) = \lambda_{j}(\eta) + \epsilon \kappa_j(\eta), \ j =1, \ldots, m, 
$$
and $\mathcal{N}_{\epsilon} = \mathcal{N}_{L} + \epsilon\mathcal{N}_{Q}$. Hence, for commutative perturbations  we have the following:
\begin{theorem}\label{prop-perturbation-1}
	Admit $[L,Q] = 0$. Then, $L_\epsilon$ is globally hypoelliptic if and only if there are positive constants $C_{j}$, $M_{j}$ and $R_{j}$ such that 
	\begin{equation*}
		\left| \lambda_{j}(\eta) + \epsilon \kappa_j(\eta) \right| \geq C_j |\eta|^{-M_j}, \quad  |\eta| \geq R_j,
	\end{equation*}
	for $j =1, \ldots, m$.
\end{theorem}

On the other hand,  the general case $[L,Q] \neq 0$ is much more challenging since, in general, there isn't a  unified theory for the study of $\lambda_{\epsilon,j}(\eta)$ in terms of $\lambda_{j}(\eta)$ and $\kappa_j(\eta)$. To handle  this problem, let us take a similar approach to the one taken in  \cite{AK19}, by A. Kirilov and F. de \'Avila Silva,  where the authors consider the theories developed in \cite{kato} by T. Kato and in \cite{Rellich} by F. Rellich.

The key point used there  is the following assumption: the eigenvalues and eigenvectors of $L_\epsilon(\eta)$ have an analytic expansion of type
\begin{equation}\label{eigen-expansion}
	\lambda_{\epsilon,j}(\eta) = \lambda_j(\eta) + \sum_{k=1}^{\infty} \sigma_{j,k}(\eta)  \, \epsilon^k \ \textrm{ and }
	\ v_{\epsilon,j}(\eta)= v_{j}(\eta) + \sum_{k=1}^{\infty}v_{\epsilon,j}(\eta) \, \epsilon^k,
\end{equation}
in a small ball at the origin, where $v_{j}(\eta)$ is a corresponding eigenvector of $\lambda_j(\eta)$. In particular, such analytic  expansion holds when $L_\epsilon(\eta)$ is normal in some ball $|\epsilon|< \epsilon'$, as the reader can see in \cite{Jamison54}.

The next result enables us to shed light  on  the problem of perturbations.
\begin{theorem}\label{prop-perturbation}
	Assume that for  $|\epsilon|< \epsilon_0$ the set
	$$
	Z_{\det L_\epsilon} =\{ \eta \in \mathbb{Z}^n; \det L_\epsilon(\eta) = 0  \}
	$$ 
	is finite, and that the eigenvalues of $L_\epsilon(\eta)$ can be written as in form  \eqref{eigen-expansion}, for all $\eta \in \mathbb{Z}^N$.
	
	Then,  system $L_\epsilon = L + \epsilon Q$ is globally hypoelliptic if and only if there are positive constants $C_j, M_j$ and $R_j$ such that
	\begin{equation*}
		\left|\lambda_j(\eta) + \sum_{k=1}^{\infty} \sigma_{j,k}(\eta) \epsilon^k\right| \geq C_j |\eta|^{-M_j}, \ |\eta| \geq R_j, \ j =1, \ldots, m.
	\end{equation*}
\end{theorem}

\begin{corollary}
	Let $P_{\epsilon}$   be the perturbed system
	$$
	P_{\epsilon} = D_t + L_\epsilon,  \ (t,y) \in \mathbb{T} \times \mathbb{T}^N,
	$$
	such that the set $Z_{\det P_{\epsilon}}$ is finite and the eigenvalues of $L_\epsilon(\eta)$ satisfy \eqref{eigen-expansion}, for $|\epsilon|< \epsilon_0$.
	
	Under these conditions, $P_{\epsilon}$ is globally hypoelliptic if and only if there are positive constants $C_j, M_j$ and $R_j$ satisfying
	\begin{equation*}
		\left|\tau + \lambda_j(\xi) + \sum_{k=1}^{\infty} \sigma_{j,k}(\xi) \epsilon^k\right| \geq C_j |(\tau,\xi)|^{-M_j}, \quad  |\tau| + |\xi| \geq R_j,
	\end{equation*}
	for $j =1, \ldots, m$.

\end{corollary}

\begin{remark}
	It is important to point out that any application of  Theorem \ref{prop-perturbation} necessarily involves
	the calculations of the coefficients $\sigma_{j,k}(\eta)$. An approach concerning this process can be found in  Chapter two in \cite{kato} and Section 6 in \cite{AK19}.

\end{remark}

It is possible to illustrate the previous results with an adaptation of Example 6.4 in \cite{AK19}.

\begin{example}
	Consider  the operators 
	$$
	\lambda_j(D)u(t,x) = \sum_{\tau \in \mathbb{Z}, \, \ell \in \mathbb{N}}{e^{i (t,x) \cdot (\tau, \ell)} (-1)^{j}\omega \ell \, \widehat{u}(\tau, \ell)}, \ j=1,2,
	$$
	defined on $\mathbb{T}^2$, for some $\omega \in \mathbb{C}$ fixed, and 
	$$
	q(D_x)v(x) = \sum_{\ell \in \mathbb{N}}{e^{i x \cdot \ell} q(\ell) \, \widehat{v}(\ell)},
	$$
	an operator on $\mathbb{T}_x$, of order $0\leq \nu <1$.
	
	Let $\Lambda (D)$ and $Q (D)$ be the matrix operators	
	\begin{equation*}
		\Lambda (D) = diag (\lambda_1(D), \lambda_2(D))
		\ \textrm{ and } \
		Q (D) = \left [
		\begin{array}{cc}
			0  &  q(D_x)\\[2mm]
			q(D_x)  & 0
		\end{array} \right],
	\end{equation*}
	and set $L_{\epsilon} = \Lambda(D) + \epsilon Q(D)$.
	
	Thus, the eigenvalues of $L_{\epsilon}(\ell)= \Lambda(\ell) + \epsilon Q(\ell)$ are 
	\begin{equation*}
		\lambda_{\epsilon,1}(\ell) = - \omega \ell - \dfrac{(q(\ell))^2}{2 \omega \ell} \, \epsilon^2 + \dfrac{(q(\ell))^4}{(2 \omega \ell)^3} \, \epsilon^4
		+ O(\epsilon^6),
	\end{equation*}
	and
	\begin{equation*}
		\lambda_{\epsilon,2}(\ell) = \omega \ell + \dfrac{(q(\ell))^2}{2 \omega \ell} \, \epsilon^2 - \dfrac{(q(\ell))^4}{(2 \omega \ell)^3} \, \epsilon^4
		+ O(\epsilon^6).
	\end{equation*}
	
	In particular,   there exists  a positive constant  $\epsilon_1$ such that $L_\epsilon$ is globally hypoelliptic for $|\epsilon|<\epsilon_1$, as the reader can see in \cite{AK19}.
	
\end{example}

%%%%%%%%%%%%%%%%%%%%%%%%%%%%%%%%%%%%%%%%%%%%%%%%%%%%%%%%%%%%%%%%
%%%%%%%%%%%%%%%%%%%%%%%%%%%%%%%%%%%%%%%%%%%%%%%%%%%%%%%%%%%%%%%%
\section{Systems with variable coefficients}\label{sec-3}
%%%%%%%%%%%%%%%%%%%%%%%%%%%%%%%%%%%%%%%%%%%%%%%%%%%%%%%%%%%%%%%%
%%%%%%%%%%

This Section addresses the systems $P = D_t  + Q(t,D_x)$, where $Q(t,D_x)$ is an $m \times m$ matrix operator with symbols $Q(t,\xi) = [c_{j,k}(t)Q_{j,k}(\xi)]$. Let us recall that equation $Pu = f \in C_{m}^{\infty}(\mathbb{T}^{n+1})$, with $u \in \mathcal{D}'_{m}(\mathbb{T}^{n+1})$, is equivalent to   
\begin{equation}\label{general-system}
	D_t \widehat{u}(t, \xi)  + Q(t, \xi) \widehat{u}(t, \xi) =\widehat{f}(t, \xi), \ t \in \mathbb{T}, \ \xi \in \mathbb{Z}^n,
\end{equation}
where $\widehat{u}(t, \xi)$ and $\widehat{f}(t, \xi)$ denotes the $m$-dimensional vectors 
$$
\widehat{u}(t, \xi) = (\widehat{u}_1(t, \xi), \ldots, \widehat{u}_m(t, \xi))
\ \textrm{ and } \
\widehat{f}(t, \xi) = (\widehat{f}_1(t, \xi), \ldots, \widehat{f}_m(t, \xi)).
$$

We analyze  the global  hypoellipticity of $P$ by studying 
the coordinates $\widehat{u}_j(t, \xi)$ in  the sense of Proposition \ref{prop-smooth}. The starting point is the introduction of \textit{strongly triangularizable symbols} which enables the study of  a suitable triangular system associated to \eqref{general-system}. 

Firstly, the analysis is restricted to symbols having smooth eigenvalues and satisfying a polynomial growth.

\begin{definition}
	We say that $Q(t,\xi)$ satisfies condition ($\mathscr{A}$) if, for each $k \in \{1, \ldots, m\}$, the eigenvalues $\lambda_k(t,\xi)$ fulfill the following conditions:

	\begin{enumerate}
		\item [(a)] they belong to the space
		$
		C^{\infty}(\mathbb{T} \times \mathbb{Z}^n) \doteq 
		\{f: \mathbb{T} \times \mathbb{Z}^n \to \mathbb{C}; \  f(\cdot, \xi) \in 	C^{\infty}(\mathbb{T}),  \ \forall \xi \in \mathbb{Z}^n\}.
		$

		\item [(b)] given $\alpha \in \mathbb{Z}_{+}$, there exists $C=C(\alpha, k)$ and 
		$\mu_k=\mu_k(\alpha)$ such that 
		\begin{equation}\label{poly-growth}
			\sup_{t\in \mathbb{T}}|\partial_t ^{\alpha} \lambda_k(t, \xi)| \leq C|\xi|^{\mu_k}, \ |\xi| \rightarrow \infty.
		\end{equation}

	\end{enumerate}
	
\end{definition}

\begin{remark}
	Along the text we use  expression ``a smooth function on $\mathbb{T} \times \mathbb{Z}^n$''  to indicate that such function belongs to  $C^{\infty}(\mathbb{T} \times \mathbb{Z}^n)$. Analogously, we say that a matrix $M(t,\xi)=[m_{j,k}(t,\xi)]_{m \times m}$ is smooth on $\mathbb{T} \times \mathbb{Z}^n$ when  $m_{j,k} \in C^{\infty}(\mathbb{T} \times \mathbb{Z}^n)$.
	
\end{remark}

\begin{remark}
	We point out that the study of differentiability for  eigenvalues (and eigenvectors) is a highly non  trivial problem (even for holomorphic matrices). For more information on this subject, see Sections 5.4 and 5.5 in \cite{kato}. Regarding  condition \eqref{poly-growth}, let us  observe that many classes of examples consist in  systems given by  $Q(t,D_x) = Q(D_x)A(t)$, where $Q(D_x)$ is a pseudo-differential operator on $\mathbb{T}^n$, of order $\nu$, and $A(t)$ is a smooth  $m\times m$ matrix with  smooth eigenvalues $\sigma_{1}(t), \ldots, \sigma_{m}(t)$. In this case, for each $k \in \{1, \ldots, m\}$, we obtain 
	$$
	\sup_{t \in \mathbb{T}} |\partial^{\alpha}_t \lambda_k(t,\xi)| =  |q(\xi)| 
	\max_{t \in \mathbb{T}} |\partial^{\alpha}_t \sigma_{k}(t)|
	\leq  C_{Q,k,\alpha} |\xi|^{\nu}.
	$$
\end{remark}

\begin{definition}\label{def-stro-tri}
	Let $Q(t,\xi)$ be a symbol satisfying condition ($\mathscr{A}$). We say that $Q(t,\xi)$  is 	strongly triangularizable if it  satisfies the following conditions:
	
	\begin{enumerate}
		\item [($\mathscr{B}_1$)] for each $\xi \in \mathbb{Z}^n$, there exists a triangularization 
		\begin{equation*}
			S^{-1}(t, \xi)  Q(t,\xi) S(t, \xi)  = \Lambda(t,\xi) + \mathcal{N}(t, \xi),
		\end{equation*}
		where 
		$
		\Lambda(t,\xi) = diag(\lambda_{1}(t,\xi), \ldots, \lambda_{m}(t,\xi))
		$
		and $\mathcal{N}(t,\xi)$ is a smooth upper triangular matrix 
		$$
		\mathcal{N}(t,\xi) =
		\left[
		\begin{array}{ccccc}
			0 & r_{1,2}(t,\xi) &\ldots & r_{1,m}(t,\xi)\\
			0 & 0      &  \ldots & r_{2,m}(t,\xi) \\
			\vdots     &  \vdots & \vdots & \vdots \\
			0 & 0      &  0 &  r_{m-1,m}(t,\xi) \\
			0 & 0      &  \ldots & 0 \\
		\end{array}
		\right].
		$$

		\item [($\mathscr{B}_2$)] the matrices $S$ and $S^{-1}$ are smooth on $\mathbb{T}\times \mathbb{Z}^n$ and, for any $\alpha \in \mathbb{Z}_+$, there exist positive constants $C$, $R$  and $\gamma$,  such that 
		\begin{equation}\label{poly-decay-S-sec4}
			\sup_{t \in \mathbb{T}}\|\partial_t^{\alpha} S(t, \xi)\|_{\mathbb{C}^{m \times m}} \leq C|\xi|^{\gamma} \ \textrm{ and } \ 
			\sup_{t \in \mathbb{T}}\|\partial_t^{\alpha} S^{-1}(t, \xi)\|_{\mathbb{C}^{m \times m}}  \leq C|\xi|^{\gamma},
		\end{equation}
		for all $|\xi| \geq R$.

		\item [($\mathscr{B}_3$)] given  $\alpha \in \mathbb{Z}_+$ and $N>0$, there are positive constants $C$ and $R$ such that
		\begin{equation*}
			\sup_{t \in \mathbb{T}}	\|\partial_t^{\alpha} B(t, \xi)\|_{\mathbb{C}^{m \times m}}  \leq C|\xi|^{-N},	 \ \forall  |\xi| \geq R,
		\end{equation*}
		where $B(t, \xi)  = S^{-1}(t, \xi)\cdot  D_t S(t, \xi)$.

	\end{enumerate}

	Additionally, $Q(t,\xi)$ is  strongly triangularizable with diagonal \textit{bounded from below} if it satisfies the following extra condition:
	\begin{enumerate}
		\item [($\mathscr{B}_4$)] for each $k \in \{1, \ldots, m\}$, there exists a positive  constant $\theta_k$ such that
		\begin{equation*}
			\Im \lambda_k(t, \xi) \geq -\theta_k, \ t \in \mathbb{T}, \  \forall \xi \in \mathbb{Z}^n.
		\end{equation*}
	\end{enumerate}

	Finally, a system as \eqref{general-system}  with a strongly triangularizable symbol (with diagonal bounded from below)
	is said to be a strongly triangularizable system	(with diagonal bounded from below).
	
\end{definition}

\begin{remark}
	Section \ref{sec-4} presents sufficient conditions for $Q(t,\xi)$ to be strongly triangularizable. It also presents a process for the construction of matrices $S$, $S^{-1}$ and $\mathcal{N}$.
\end{remark}

\begin{remark}
	The requirement of conditions \eqref{poly-growth} and ($\mathscr{B}_4$)  will be clarified by Theorem \ref{tec-theorem}.
\end{remark}

Now, the investigation proceeds by assuming that system \eqref{general-system} is strongly triangularizable. Firstly, it is important to  emphasize the following properties:

\begin{enumerate}
	\item [(i)] condition $(\mathscr{B}_1)$ implies that \eqref{general-system} can be rewritten in the form
	\begin{equation}\label{triangular-system}
		(D_t + \Lambda(t,\xi) + \mathcal{N}(t,\xi) + B(t, \xi)  )v(t,\xi)  = g(t,\xi),
	\end{equation}
	by setting
	$$
	\widehat{u}(t,\xi) = S(t,\xi)v(t, \xi)  \ \textrm{ and } \ 
	\widehat{f}(t,\xi) = S(t,\xi)g(t, \xi);
	$$

	\item[(ii)] it follows from conditions ($\mathscr{B}_1$), ($\mathscr{B}_2$) and  \eqref{poly-growth}  that the  entries $r_{j,k}(t,\xi)$ of $\mathcal{N}$ have all derivatives bounded by some polynomial, namely, for every $\alpha \in \mathbb{Z}_{+}$, there exists $C=C(\alpha, j,k)$ and $\mu_{j,k}=\mu_{j,k}(\alpha)$ such that 
	\begin{equation}\label{poly-growth-nilpotent}
		\sup_{t\in \mathbb{T}}|\partial_t ^{\alpha} r_{j,k}(t, \xi)| \leq C|\xi|^{\mu_{j,k}}, \ |\xi| \rightarrow \infty;
	\end{equation}
	
	\item[(iii)] in view of condition ($\mathscr{B}_2$), we  obtain that a coordinate of the vector $v(t, \xi)$ satisfies  \eqref{part-smooth-coef} if and only if the corresponding coordinate of $\widehat{u}(t,\xi)$ also satisfies \eqref{part-smooth-coef}.
	
	\item [(iv)] we can replace the study of solutions of \eqref{general-system} by the study of solutions of  \eqref{triangular-system}.
	
\end{enumerate}

Next, we show that the matrices  $B(t,\xi) = [b_{j,k}(t,\xi)]_{m \times m}$ play no role on the regularity of solutions of system \eqref{triangular-system}. To do this, consider the space $\mathscr{S}(\mathbb{T}\times \mathbb{Z}^n)$ of all functions $a \in C^{\infty}(\mathbb{T}\times \mathbb{Z}^n)$ satisfying the following: $\forall \alpha \in \mathbb{Z}_+$ there exists  $N \in \mathbb{R}$ and positive constants $C,R$ such that
\begin{equation*}
	\sup_{t \in \mathbb{T}}	|\partial_t^{\alpha} a(t, \xi)|  \leq C|\xi|^{N},	 \ \forall  |\xi| \geq R.
\end{equation*}

Therefore, for any $a\in \mathscr{S}(\mathbb{T}\times \mathbb{Z}^n)$ it is well defined the operator $a(t,D_x): \mathcal{D}'(\mathbb{T}^{n+1}) \to \mathcal{D}'(\mathbb{T}^{n+1})$ given by
\begin{equation}\label{Fourier_Operator}
	a(t,D_x) w = \sum_{\xi \in \mathbb{Z}^N}a(t,\xi)\widehat{w}(t,\xi) \exp(i x \cdot \xi).	
\end{equation}
In particular, it follows that   $a(t,D_x)(C^{\infty}(\mathbb{T}^{n+1})) \subset C^{\infty}(\mathbb{T}^{n+1})$.

Now, note that all functions 
$\lambda_{k}, r_{j,k}$, $b_{j,k}$ belongs to $\mathscr{S}(\mathbb{T}\times \mathbb{Z}^n)$ and we may define  operators $\lambda_{j,k}(t,D_x)$, $r_{j,k}(t,D_x)$ and $b_{j,k}(t,D_x)$ as in \eqref{Fourier_Operator}. Moreover, we have the following:

\begin{lemma}\label{regularizing}
	Operators $b_{j,k}(t,D_x)$	are smoothing, namely, 
	$$
	b_{j,k}(t,D_x) (\mathcal{D}'(\mathbb{T}^{n+1})) \subset C^{\infty}(\mathbb{T}^{n+1}).
	$$
\end{lemma}	

\begin{proof}
	Let $w \in \mathcal{D}'(\mathbb{T}^{n+1})$ be a distribution, $\alpha \in \mathbb{Z}_+$ and $N>0$. Given $0\leq\gamma \leq \alpha$, it  follows from Proposition \ref{prop-smooth} the existence of constants $M,C_1,R_1>0$ such that
	$\sup_{t \in \mathbb{T}}|\partial^{\gamma}_t \widehat{w}(t,\xi)| \leq C_1|\xi|^M$, for all $|\xi|\geq R_1$. Now, for $\widetilde{N}=N+M$, we obtain constants $C_2,R_2>0$ satisfying
	$$
	\sup_{t \in \mathbb{T}}|\partial_t^{\alpha-\gamma}  b_{j,k}(t,\xi)| \leq C_2|\xi|^{-\widetilde{N}}, \ \forall |\xi|\geq R_2,
	$$
	in view of condition $(\mathscr{B}_3)$.
	
	Hence,
	\begin{align*}
		|\partial^{\alpha}_t [b_{j,k}(t,\xi)\widehat{w}(t,\xi)] |
		& \leq 
		\sum_{\gamma \leq \alpha} 
		\binom{\alpha}{\gamma}|\partial_t^{\alpha-\gamma}  b_{j,k}(t,\xi)| \, |\partial_t^{\gamma}  \widehat{w}(t,\xi)| \\
		& \leq C |\xi|^{-N},
	\end{align*}
	for all $|\xi|\geq \max\{R_1,R_2\}$. Then, $b_{j,k}(t,D_x)w \in C^{\infty}(\mathbb{T}^{n+1})$.
	
\end{proof}

Next, we show that a strongly triangularizable system can be reduced into a triangular form.

\begin{theorem}\label{rediction-theorem}
	Suppose that  $Q(t,\xi)$  is a strongly triangularizable symbol and set
	$$
	\Lambda(t,D_x) = [\lambda_{j,k}(t,D_x)]_{m \times m}, \ 
	\mathcal{N}(t,D_x) = [r_{j,k}(t,D_x)]_{m \times m} \  \textrm{ and } \
	B(t,D_x) = [b_{j,k}(t,D_x)]_{m \times m}.
	$$

	Then, the following statements are equivalent:
	\begin{enumerate}
		\item [a)] system $P = D_t + Q(t,D_x)$ is  globally hypoelliptic;
		
		\item [b)] system $L = D_t + \Lambda(t,D_x) + \mathcal{N}(t,D_x) +  B(t,D_x)$  is  globally hypoelliptic;
		
		\item [c)] system $T= D_t + \Lambda(t,D_x) + \mathcal{N}(t,D_x)$ is  globally hypoelliptic;		
	\end{enumerate}
	
\end{theorem}

\begin{proof}
	$(a)\Longrightarrow (b)$. Let   $v \in \mathcal{D}'_m(\mathbb{T}^{n+1})$ be a solution of $Lv = g \in C^{\infty}_m(\mathbb{T}^{n+1})$ and consider 
	\begin{equation*}
		u = \sum_{\xi \in \mathbb{Z}^N} S(t,\xi) \widehat{v}(t,\xi) \exp(i x \cdot \xi) \ \textrm{ and } \ 
		f = \sum_{\xi \in \mathbb{Z}^N} S(t,\xi) \widehat{g}(t,\xi) \exp(i x \cdot \xi).
	\end{equation*}
	
	Clearly, $u \in \mathcal{D}'_m(\mathbb{T}^{n+1})$ and $f \in C^{\infty}_m(\mathbb{T}^{n+1})$. Moreover, since
	$$
	S(t,\xi)B(t,\xi) = D_tS(t,\xi) \ \textrm{ and } \ Q(t,\xi) S(t,\xi) = S(t,\xi)[\Lambda(t,\xi) + \mathcal{N}(t,\xi)],
	$$
	we obtain
	\begin{align*}
		Pu & = \sum_{\xi \in \mathbb{Z}^N}\left\{ D_t[S(t,\xi) \widehat{v}(t,\xi)] + Q(t,\xi) S(t,\xi) \widehat{v}(t,\xi) \right\}\exp(i x \cdot \xi) \\
		& = \sum_{\xi \in \mathbb{Z}^N}\left\{ [D_tS(t,\xi)] \widehat{v}(t,\xi) + S(t,\xi) [D_t \widehat{v}(t,\xi) + (\Lambda(t,\xi) + \mathcal{N}(t,\xi) ) \widehat{v}(t,\xi)] \right\}\exp(i x \cdot \xi) \\
		& = \sum_{\xi \in \mathbb{Z}^N}\left\{ S(t,\xi)[B(t,\xi)\widehat{v}(t,\xi) + D_t \widehat{v}(t,\xi) + (\Lambda(t,\xi) + \mathcal{N}(t,\xi) ) \widehat{v}(t,\xi)] \right\}\exp(i x \cdot \xi) \\
		& = \sum_{\xi \in \mathbb{Z}^N} S(t,\xi) \widehat{g}(t,\xi) \exp(i x \cdot \xi) \\
		& = f.
	\end{align*}
	
	Hence, $u \in  C^{\infty}_m(\mathbb{T}^{n+1})$ which imply 
	$v \in  C^{\infty}_m(\mathbb{T}^{n+1})$. Therefore, $L$  is  globally hypoelliptic.

	$(b)\Longrightarrow (a)$.  We can use a similar argument, however it is necessary to observe that 
	$$
	D_tS^{-1}(t,\xi) + B(t,\xi)S^{-1}(t,\xi)\equiv 0,
	$$
	since $D_t[S^{-1}(t,\xi)S(t,\xi)] \equiv 0$. Then, given $u \in \mathcal{D}'_m(\mathbb{T}^{n+1})$ such that $Pu = f \in C^{\infty}_m(\mathbb{T}^{n+1})$
	we define 
	\begin{equation*}
		v = \sum_{\xi \in \mathbb{Z}^N} S^{-1}(t,\xi) \widehat{u}(t,\xi) \exp(i x \cdot \xi) \ \textrm{ and } \ 
		g = \sum_{\xi \in \mathbb{Z}^N} S^{-1}(t,\xi) \widehat{f}(t,\xi) \exp(i x \cdot \xi).
	\end{equation*}
	
	In this case, we have $Lv = g$  and then $v \in  C^{\infty}_m(\mathbb{T}^{n+1})$. Therefore, we obtain 
	$u \in  C^{\infty}_m(\mathbb{T}^{n+1})$ and the global hypoellipticity of $P$.

	$(b)\Longrightarrow (c)$. First, note that
	$$
	B(t,D_x) (\mathcal{D}'_m(\mathbb{T}^{n+1})) \subset C^{\infty}_m(\mathbb{T}^{n+1}),
	$$
	in view of Lemma \ref{regularizing} 
	
	Now, consider $u \in\mathcal{D}'_m(\mathbb{T}^{n+1})$ such that $Tu \in C^{\infty}_m(\mathbb{T}^{n+1})$. Then, 
	$$
	Lu =  [Tu + B(t,D_x)u]  \in C^{\infty}_m(\mathbb{T}^{n+1}) \Longrightarrow 
	u \in C^{\infty}_m(\mathbb{T}^{n+1}),
	$$
	implying  that  $T$  is  globally hypoelliptic.

	$(c)\Longrightarrow (b)$. Let $u \in\mathcal{D}'_m(\mathbb{T}^{n+1})$ satisfying $Lu \in C^{\infty}_m(\mathbb{T}^{n+1})$. Then, 
	$$
	Tu = Lu - B(t,D_x)u   \Longrightarrow 
	Tu \in C^{\infty}_m(\mathbb{T}^{n+1}) 
	\Longrightarrow
	u \in C^{\infty}_m(\mathbb{T}^{n+1}),
	$$
	hence, $L$  is  globally hypoelliptic.

\end{proof}

\begin{example}\label{exemple_ST}
	Consider $P = D_t + Q(t, D_x)$ given by
	$$
	Q (t,D_x) = \left [
	\begin{array}{cc}
		a(t)  &  b(t) D_x\\[2mm]
		b(t) D_x  & a(t)
	\end{array} \right], \ (t,x) \in  \mathbb{T}^2,
	$$
	where $a,b$ are  smooth real-valued functions and $b \neq 0$.
	
	The eigenvalues of $Q(t,\xi)$ are $\lambda(t,\xi) = a(t) \pm b(t)\xi$ and we can choose $h(t,\xi)\equiv (1 , 1)$ as the eigenvector associated to $\lambda_1 (t,\xi) = a(t) + b(t)\xi$. Note that, by choosing
	$$
	S(t,\xi) = 	
	\left [
	\begin{array}{cc}
		1  &  0\\[2mm]
		1  & 1
	\end{array} 
	\right]
	\ \textrm{ and } \ 
	S^{-1}(t,\xi) =
	\left [
	\begin{array}{rr}
		1  &  0\\[2mm]
		-1   & 1
	\end{array} 
	\right],
	$$
	we obtain 
	$$
	S^{-1}(t,\xi)
	Q(t,\xi)
	S(t,\xi)
	=
	\left [
	\begin{array}{cc}
		\lambda_1 (t,\xi)  &   b(t) \xi\\[2mm]
		0  & \lambda_2 (t,\xi)
	\end{array} \right].
	$$
	
	Therefore, $Q (t,\xi)$ is strongly triangularizable with diagonal bounded from below, since its eigenvalues are real-valued functions and $B(t, \xi) \equiv 0$, $\forall \xi \in \mathbb{Z}^n$. In this case,
	$$
	\Lambda(t,D_x) = 
	\left [
	\begin{array}{cc}
		a(t) + b(t)D_x  &   0 \\[2mm]
		0  & a(t) - b(t)D_x
	\end{array} \right] \ \textrm{ and } \ 
	\mathcal{N}(t,D_x) =	\left [
	\begin{array}{cc}
		0   &   b(t) D_x\\[2mm]
		0  & 0
	\end{array} \right].
	$$
	
\end{example}

\subsection{The study of triangularizable systems}

Admit that $P = D_t + Q(t,D_x)$ is strongly triangularizable,  let 
$u \in \mathcal{D}'_{m}(\mathbb{T}^{n+1})$ be a solution of 
$P u = f \in {C}^{\infty}_{m}(\mathbb{T}^{n+1})$. In view of  Theorem \ref{rediction-theorem},  it is sufficient to study system
$$
D_t v(t, \xi) + \Lambda(t,\xi)v(t, \xi)  + \mathcal{N}(t,\xi)v(t, \xi) = g(t, \xi),  
$$
or equivalently, 
\begin{equation}\label{general-traing-sys}
	\left\{
	\begin{array}{r}
		D_tv_1(t, \xi) + \lambda_{1}(t,\xi)v_1(t, \xi)  + \ldots + r_{1,m}(t, \xi) v_m(t, \xi) = g_1(t, \xi)\\
		D_tv_2(t, \xi)\lambda_{2}(t,\xi)v_2(t, \xi)  + \ldots + r_{2,m}(t, \xi) v_m(t, \xi) = g_2(t, \xi) \\
		\vdots \\
		D_tv_m(t, \xi) + \lambda_{m}(t,\xi)v_m(t,\xi) = g_m(t,\xi) \\
	\end{array}
	\right.,
\end{equation}
where $\widehat{u}(t,\xi) = S(t,\xi)v(t, \xi)$ and $\widehat{f}(t,\xi) = S(t,\xi)g(t, \xi)$.

Since the  terms $D_tv_k(t, \xi) + \lambda_{k}(t,\xi)v_k(t, \xi)$ (in each line of \eqref{general-traing-sys}) play an important role in this analysis, then let us introduce the following family of operators
$$
\mathscr{L}_k = \{\mathscr{L}_k^{\xi} = D_t + \lambda_k(t,\xi), \ \xi \in \mathbb{Z}^n \},
$$
for each $k \in \{1, \ldots, m\}$, as well as the following definition:

\begin{definition}
	We say that the family  $\mathscr{L}_k$ is globally hypoelliptic  if the following conditions hold: for every $\omega \in \mathcal{D}'(\mathbb{T}^{n+1})$  and $h \in C^{\infty}(\mathbb{T}^{n+1})$ satisfying the equations
	\begin{equation}\label{Lk_equation-h}
		\mathscr{L}_k^{\xi} \, \widehat{\omega}(t,\xi) = \widehat{h}(t,\xi),  \ \forall \xi \in \mathbb{Z}^n,
	\end{equation}
	we have $\omega \in C^{\infty}(\mathbb{T}^{n+1})$. 
\end{definition}

For each $k \in \{1, \ldots, m\}$ we define the operators (acting on $C^{\infty}(\mathbb{T}^n)$)
\begin{equation*}
	\lambda_{0,k}(D_x) w(x) =  \sum_{\xi \in \mathbb{Z}^n}{e^{i x \cdot \eta} \lambda_{0,k}(\xi) \widehat{w}(\xi)},
\end{equation*}
and 
\begin{equation*}
	\mathscr{L}_{0,k} = D_t + \lambda_{0,k}(D_x),  
\end{equation*}
where 
$\lambda_{0,k}(\xi) = (2\pi)^{-1}\int_{0}^{2\pi}\lambda_{k}(t, \xi)dt.$

Let us  recall that, by Theorem \ref{AGKM}, the operator $\mathscr{L}_{0,k}$ is globally hypoelliptic
if and only if there exist positive constants $C$, $M$ and $R$ satisfying
$$
|\tau+\lambda_{0,k}(\xi)|\geq C(|\tau| + |\xi|)^{-M}, \ \textrm{for all} \ |\tau| + |\xi| \geq R,
$$
or equivalently, there exist positive constants $\widetilde{C}$, $\widetilde{M}$ and $\widetilde{R}$ such that
\begin{equation}\label{lambda_smooth_est}
	|1-e^{\pm2\pi i\lambda_{0,k}(\xi)}|\geq \widetilde{C}|\xi|^{-\widetilde{M}}, \ \textrm{for all} \  |\xi|\geq \widetilde{R}.
\end{equation}

The next result exhibits a necessary condition for the global hypoellipticity of a family  $\mathscr{L}_k$.

\begin{proposition}\label{Z_k-is-finite}
	If  $\mathscr{L}_k$ is globally hypoelliptic, then the set  $Z_{k} = \{ \xi \in \mathbb{Z}^n; \ \lambda_{0,k}(\xi)  \in \mathbb{Z}  \}$ is finite.
\end{proposition}

\begin{proof}
	Suppose that $Z_{k}$ is infinite an let $\{\xi_{\ell}\}_{\ell \in \mathbb{N}} \subset Z_{k}$ be an increasing sequence. For each $\ell \in \mathbb{N}$, choose $t_\ell \in [0,2\pi]$
	satisfying
	$$
	\int_{0}^{t_\ell} \Im \lambda_{k}(r,\xi_{\ell}) dr = \max_{t \in \mathbb{T}} 
	\int_{0}^{t} \Im \lambda_{k}(r,\xi_{\ell}) dr,
	$$
	and, by setting
	$
	\kappa_{\ell} \doteq \exp \left( -  \int_{0}^{t_\ell} \Im \lambda_{k}(r,\xi_{\ell}  )dr \right),
	$
	define
	$$
	\widehat{u}(t, \xi) = 
	\left\{
	\begin{array}{l}
		\kappa_{\ell} \exp \left( - i \int_{0}^{t}  \lambda_{k}(r,\xi_{\ell} )dr \right), \ \textrm{ if } \ \xi = \xi_\ell,\\
		0, \ \textrm{ if } \ \xi \neq \xi_\ell.
	\end{array}
	\right.
	$$

	Since  $\xi_{\ell} \in Z_k$, it follows that $\widehat{u}(\cdot, \xi)$ are smooth and $2\pi$-periodic.  Moreover, $\{\widehat{u}(t, \xi)\}_{\xi \in \mathbb{Z}^n}$ defines a distribution $u \in \mathcal{D}'(\mathbb{T}^{n+1})  \setminus C^{\infty}(\mathbb{T}^{n+1})$ such that
	$\mathscr{L}_k^{\xi}  \widehat{u}(t, \xi) = 0$, for every $\xi \in \mathbb{Z}^n$. Hence, the family $\mathscr{L}_k$ is not globally hypoelliptic.

\end{proof}

Now,  the next result presents  a first connection between the analysis of the system $P = D_t + Q(t,D_x)$ and the related constant coefficient operators $\mathscr{L}_{0,k}$.

\begin{theorem}\label{tec-theorem}
	Admit that  $\{\lambda_{k}(t,\xi)\}$ satisfies  ($\mathscr{B}_4$).  
	Then,  $\mathscr{L}_{0,k}$ is globally hypoelliptic if and only if the family $\mathscr{L}_k$ is globally hypoelliptic.
\end{theorem}

\begin{remark}
	It is important to  emphasize that  the proof of this result is a slight modification of the proofs of Theorem 3.5 in \cite{AGKM} and Theorem 3.3 in \cite{Avila}.
\end{remark}

\begin{proof}
	Assume that $\mathscr{L}_{0,k}$ is globally hypoelliptic and consider $\omega$ and $h$  satisfying equations \eqref{Lk_equation-h}. It follows from Theorem \ref{AGKM}, part (a), that  $Z_{k}$ is finite. Hence,  the solutions of \eqref{Lk_equation-h}  can be written in the form
	\begin{equation*}
		\widehat{\omega}(t, \xi) = \frac{1}{e^{ 2 \pi i\lambda_{0,k}(\xi)} - 1} \int_{0}^{2\pi}\exp\left(i\int_{t}^{t+s}\!\!\lambda_k(r, \, \xi)\, dr\right) \widehat{h}(t+s, \xi)ds,
	\end{equation*}
	for $|\xi|$ large enough.

	Now, let   $\alpha$ be a non-negative integer and consider $N_1>0$. By applying the Leibniz formula and estimates \eqref{part-smooth-coef}, \eqref{poly-growth} and  \eqref{lambda_smooth_est} we obtain  positive constants  $C_1$ and $R_1$ such that
	\begin{equation*}
		|\partial_t^{\alpha}  \omega(t, \xi)|  \leq C_1|\xi|^{- N_1 + \theta_k}
		\int_{0}^{2\pi}\exp\left(-\int_{t}^{t+s}\!\!\Im {\lambda_k(r, \, \xi)} \, dr\right) ds, \ 
		|\xi|\geq R_1.
	\end{equation*}
	
	The exponential term in the last integral is bounded in view of condition ($\mathscr{B}_4$), thus  $\omega$ is a smooth function on $\mathbb{T}^{n+1}$.

	For the converse, admit that $\mathscr{L}_{0,k}$ is not globally hypoelliptic. Thus, in view of \eqref{lambda_smooth_est}, there is a sequence $\{\xi_{\ell}\}_{\ell \in \mathbb{N}}$ such that $|\xi_{\ell}|$ is strictly increasing, $|\xi_{\ell}|>\ell$ and \begin{equation*}
		|1-e^{-2\pi i\lambda_{0,k}(\xi_{\ell})}|<|\xi_{\ell}|^{-\ell}, \ \textrm{for all} \ \ell \in\mathbb{N}.
	\end{equation*}
	
	If $Z_{k}$ is infinite, then $\mathscr{L}_{k}$ is not globally hypoelliptic in view of Proposition \ref{Z_k-is-finite}, thus we can proceed by assuming that
	$Z_{k}$ is finite and, in particular, we admit $\xi_\ell\not\in Z_{k}$, for all $\ell \in \mathbb{N}$. 
	
	For each $\ell,$ we may choose $t_\ell\in[0,2\pi]$ so that $\int_{t_\ell}^{t}\Im\lambda_k(r,\xi_\ell)dr\leqslant 0,$ for all $t\in[0,2\pi].$
	Indeed, for all $t\in[0,2\pi]$ we can write
	$$
	\int_{t_\ell}^{t}\Im\lambda_k(r,\xi_\ell)dr=\int_{0}^{t}\Im\lambda_k(r,\xi_\ell)dr-\int_{0}^{t_\ell}\Im\lambda_k(r,\xi_\ell)dr
	$$
	and we take $t_\ell$ satisfying $$
	\int_{0}^{t_\ell}\Im\lambda_k(r,\xi_\ell)dr=\max_{t\in[0,2\pi]}\int_{0}^{t}\Im\lambda_k(r,\xi_\ell)dr.
	$$
	
	Passing to a subsequence, if necessary, we can assume that $t_\ell\rightarrow t_0,$ for some $t_{0}\in[0,2\pi]$. Consider $I$ to be a closed interval in $(0,2\pi)$ such that $t_0\not\in I$ and let $\phi$ be a real-valued function compactly supported in $I$,  such that $0\leqslant \phi(t)\leqslant 1$ and
	$\int_{0}^{2\pi}\phi(t)dt>0.$

	For each $\ell,$ let $\widehat{f}(\cdot,\xi_\ell)$ be a $2\pi-$periodic extension of
	$$
	(1-e^{-2\pi  i\lambda_{0,k}(\xi_\ell)})\exp\left(-\int_{t_\ell}^{t}i\lambda_k(r,\xi_\ell)dr\right)\phi(t).
	$$

	Since $\lambda_{0,k}(\xi)$ is bounded by $|\xi|^{\nu_k}$ and $\int_{t_n}^{t}\Im\lambda_k(r,\xi_n)dr\leqslant 0$, for all $t\in[0,2\pi],$ 
	we obtain a function $f \in C^{\infty} (\mathbb{T}^{n+1})$ defined by
	$$
	f(t,x) = 
	\left\{
	\begin{array}{l}
		\sum_{n=1}^{\infty}\widehat{f}(t,\xi_n)e^{ix\xi_n}, \ \xi = \xi_{\ell}, \\
		0, \  \xi \neq \xi_{\ell}.
	\end{array}
	\right.	
	$$
	
	On the other hand, straightforward calculations show  that $u = \sum_{\xi \in \ \mathbb{Z}^n} \widehat{u}(t,\xi) e^{i x \cdot \xi}$
	belongs to $\mathcal{D}'(\mathbb{T}^{n+1})\setminus C^{\infty}(\mathbb{T}^{n+1})$, where 
	$$
	\widehat{u}(t,\xi_\ell)=\frac{1}{ 1-e^{-2\pi i\lambda_{0,k}(\xi_\ell)}}\int_{0}^{2\pi}\exp\left(-\int_{t-s}^{t}i\lambda_k(r,\xi_\ell)dr\right)\widehat{f}(t-s,\xi_\ell)ds,
	$$
	for $\xi = \xi_{\ell}$ and $\widehat{u}(t,\xi_\ell) = 0$, for $\xi \neq \xi_{\ell}$. 
	
	Since  $\mathscr{L}_k^{\xi}\widehat{u}(t,\xi)=-i\widehat{f}(t,\xi)$, for every $\xi \in \mathbb{Z}^n$, we obtain that  $\mathscr{L}_k$ is not globally hypoelliptic.

\end{proof}

\begin{remark}
	It is important to  emphasize that condition ($\mathscr{B}_4$) can be replaced by 
	\begin{equation*}
		\Im \lambda_k(t, \xi)  \leq \theta_k, \ t \in \mathbb{T}, \  \forall \xi \in \mathbb{Z}^n,
	\end{equation*}
	since  the solutions of \eqref{Lk_equation-h} have the equivalent form
	\begin{equation*}
		\omega(t, \xi) = \frac{1}{1 - e^{-  2 \pi i\lambda_{0,k}(\xi) }} \int_{0}^{2\pi}\exp\left(-i\int_{t-s}^{t}\!\! \lambda_k(r, \, \xi)\, dr\right) h(t-s, \xi)ds.
	\end{equation*}
	
	Hence, we get the boundedness 
	\begin{align*}
		\sup_{s \in[0, 2\pi]}\left|\exp\left(-i\int_{t-s}^{t}\!\!\lambda_k(r, \, \xi) \, dr\right) \right| & =  
		\sup_{s \in[0, 2\pi]}\left|\exp\left(\int_{t-s}^{t}\!\!{\Im \lambda_k}(r, \, \xi) \, dr\right) \right|  \\
		& \leq  e^{2\pi \theta_k}.
	\end{align*}	
\end{remark}

At this point it is possible to exhibit the main result.

\begin{theorem}\label{The-Nece-suff-GH}
	Suppose that  $Q(t,\xi)$  is a strongly triangularizable symbol with diagonal bounded from below.  Then, system $P = D_t + Q(t,D_x)$ is globally hypoelliptic if and only if the diagonal system
	$$
	\mathscr{L}_{0} = diag(\mathscr{L}_{0,1}, \ldots, \mathscr{L}_{0,m})
	$$
	is globally hypoellitpic.
\end{theorem}

\begin{proof}
	Let us prove the sufficiency part. For this, assume that $\mathscr{L}_{0}$ is globally hypoelliptic and let  $u \in \mathcal{D}'_{m}(\mathbb{T}^{n+1})$ be a solution of $Pu = f \in C_{m}^{\infty}(\mathbb{T}^{n+1})$.  We apply the triangularization process and rewrite the triangular system \eqref{general-traing-sys} in the form 
	\begin{equation}\label{g_tildle-Syst}
		\left\{
		\begin{array}{l}
			D_tv_k(t, \xi) + \lambda_{k}(t,\xi)v_k(t, \xi) = \widetilde{g}_k(t, \xi), \ k=1, \ldots, m-1,\\
			D_tv_m(t, \xi) + \lambda_{m}(t,\xi)v_m(t,\xi) = g_m(t,\xi),
		\end{array}
		\right.
	\end{equation}
	where
	\begin{equation*}
		\widetilde{g}_k(t, \xi) = g_{k}(t,\xi) - \sum_{j=k+1}^{m} r_{j,k} (t,\xi) v_{j}(t,\xi).
	\end{equation*}
	
	The global hypoellipticity of $\mathscr{L}_{0}$ implies that each $\mathscr{L}_{0,j}$ is  globally hypoelliptic and, by Theorem \ref{tec-theorem}, each  family $\mathscr{L}_{j}$ is also globally hypoelliptic.
	
	Given $k \in \{1, \ldots, m \}$, consider the formal series
	$$
	v_k(t,x) = \sum_{\xi \in \mathbb{Z}^n} e^{i x \cdot \xi} v_{k}(t,\xi)          
	\ \textrm{ and } \
	g_k(t,x) = \sum_{\xi \in \mathbb{Z}^n} e^{i x \cdot \xi} g_{k}(t,\xi).
	$$
	
	By the  triangularization process we have  $g_k(t,x) \in C^{\infty}(\mathbb{T}^{n+1})$, for each $k \in \{1, \ldots, m \}$. The last equation in \eqref{g_tildle-Syst} implies that $\mathscr{L}_{m}^{\xi}v_m(t,\xi)=g_{m}(t,\xi)$, for all $\xi \in \mathbb{Z}^n$. Hence, the hypoellipticity  of  $\mathscr{L}_{m}$ ensures that $v_m(t,x)$ is a smooth function on $\mathbb{T}^{n+1}$.

	Now, for $k=m-1$ we obtain
	$$
	\mathscr{L}_{m-1}^{\xi}v_{m-1}(t,\xi) = \widetilde{g}_{m-1}(t,\xi) = g_{m-1}(t,\xi) - r_{j,m-1} (t,\xi) v_{m}(t,\xi), \ \forall \xi \in \mathbb{Z}^n.
	$$
	
	Since $r_{j,\ell}$ satisfies \eqref{poly-growth-nilpotent}, then all derivatives of the term $r_{j,m-1} (t,\xi) v_{m}(t,\xi)$ converge to zero faster than any polynomial, which implies
	$$
	\widetilde{g}_{m-1}(t,x) = \sum_{\xi \in \mathbb{Z}^n} e^{i x \cdot \xi} \, \widetilde{g}_{m-1}(t,\xi) \in C^{\infty}(\mathbb{T}^{n+1}),
	$$
	hence $v_{m-1}(t,x) \in C^{\infty}(\mathbb{T}^{n+1})$, in view of the hypoellipticity  of  $\mathscr{L}_{m-1}$.

	Finally, by successive applications of these arguments,  we obtain that each component $v_k$ defines a smooth function on $\mathbb{T}^{n+1}$. Thus, $v \in C_{m}^{\infty}(\mathbb{T}^{n+1})$ and  $P$ is globally hypoelliptic.

	To prove the necessary part,  we admit that  $\mathscr{L}_{0}$ is not globally hypoelliptic. This assumption implies that at least   one of the operators $\mathscr{L}_{0,k}$ is also not globally hypoelliptic, and, consequently, neither is the family $\mathscr{L}_{k}$.

	Assume, for a moment, that $k=1$. Hence, there exists $v_1 \in \mathcal{D}'(\mathbb{T}^{n+1}) \setminus C^{\infty}(\mathbb{T}^{n+1})$  and  $g_1 \in C^{\infty}(\mathbb{T}^{n+1})$ such that
	$\mathscr{L}_{1}^{\xi} \widehat{v}_1(t,\xi) = \widehat{g}_1(t,\xi)$, for all $\xi \in \mathbb{Z}^n$.  Thus, the vector
	$
	\widehat{v}(t,\xi) = (\widehat{v}_1(t,\xi), 0, \ldots, 0)
	$
	satisfies the equation 
	$$
	(D_t + \Lambda(t,\xi) + \mathcal{N}(t,\xi)) \widehat{v}(t,\xi) = (\widehat{g}_1(t,\xi), 0, \ldots, 0),
	$$
	which implies $P$ not globally hypoelliptic.

	Now, let us assume that for some $k<m$ we have $\mathscr{L}_{1}$, $\mathscr{L}_{2}$,  $\ldots$, $\mathscr{L}_{k-1}$ globally hypoelliptic and $\mathscr{L}_{k}$  not globally hypoelliptic. Consider  $v_{k}$, $g_k$ playing the roles of $v_1$ and $g_1$ in the previous argument.

	Since $\mathscr{L}_{0,k-1}$ is globally hypoelliptic there exits $R>0$ such that
	$e^{ 2 \pi i\lambda_{0,k-1}(\xi)} \neq 1$, for $|\xi|\geq R$. Thus, we can define 
	\begin{equation*}
		\widehat{v}_{k-1}(t,\xi) = \frac{1}{e^{ 2 \pi i\lambda_{0,k-1}(\xi)} - 1} \int_{0}^{2\pi}\exp\left(i\int_{t}^{t+s}\!\!\lambda_{k-1}(r, \, \xi)\, dr\right) h_{k-1}(t+s, \xi)ds,
	\end{equation*}
	for $|\xi|\geq R$ and  $\widehat{v}_{k-1}(t,\xi) = 0$ for $|\xi|< R$, where 
	$$
	h_{k-1}(t,\xi) = -i \, r_{k-1,k}(t,\xi) \widehat{v}_k(t,\xi), \ |\xi|\geq R.
	$$
	
	In view of the hypoellipticity of $\mathscr{L}_{0,k-1}$ and the hypothesis on $v_k$, we can conclude that 
	$$
	v_{k-1}(t,x) = \sum_{\xi \in \mathbb{Z}^n} e^{i x \cdot \xi} \, \widehat{v}_{k-1}(t,\xi)
	\in \mathcal{D}'(\mathbb{T}^{n+1}).
	$$
	In particular, we have
	\begin{equation*}
		\mathscr{L}_{0,k-1}^{\xi} \widehat{v}_{k-1}(t,\xi)  =
		\left\{
		\begin{array}{l}
			- r_{k-1,k}(t,\xi) \widehat{v}_k(t,\xi), \ |\xi|\geq R, \\
			0, \ |\xi|< R,
		\end{array}
		\right.
	\end{equation*}
	and, by defining 
	$$
	\widehat{g}_{k-1}(t,\xi) = \mathscr{L}_{0,k-1}^{\xi} \widehat{v}_{k-1}(t,\xi) + 
	r_{k-1,k}(t,\xi) \widehat{v}_k(t,\xi),
	$$
	we get 
	$$
	g_{k-1}(t,x) = \sum_{\xi \in \mathbb{Z}^n} e^{i x \cdot \xi} \, \widehat{g}_{k-1}(t,\xi)
	\in C^{\infty}(\mathbb{T}^{n+1}),
	$$
	since $\widehat{g}_{k-1}(t,\xi) = 0$, for $|\xi| \geq R$.
	
	By  repeating  these arguments it is possible to construct $k-1$ distributions
	$v_1, \ldots, v_{k-1}$ and $k-1$ smooth functions $g_1, \ldots, g_{k-1}$ such that the vector
	$$
	\widehat{v}(t,\xi) = (\widehat{v}_1(t,\xi), \ldots, \widehat{v}_{k-1}(t,\xi), \widehat{v}_k(t,\xi), 0, \ldots, 0)
	$$
	satisfies 
	$$
	(D_t + \Lambda(t,\xi) + \mathcal{N}(t,\xi)) \widehat{v}(t,\xi) = (\widehat{g}_1(t,\xi), \ldots, \widehat{g}_{k-1}(t,\xi),  \widehat{g}_k(t,\xi), 0, \ldots, 0),
	$$
	which implies that $P$ is not globally hypoelliptic.
	
	The case $k=m$ is similar since it is enough to consider the vectors
	$$
	\widehat{v}(t,\xi) = (\widehat{v}_1, \ldots, \widehat{v}_{k-1}, \widehat{v}_m)
	\ \textrm{ and } \
	\widehat{g}(t,\xi)=(\widehat{g}_1, \ldots, \widehat{g}_{k-1},  \widehat{g}_m),
	$$
	for which
	$
	(D_t + \Lambda(t,\xi) + \mathcal{N}(t,\xi)) \widehat{v}(t,\xi) = \widehat{g}(t,\xi).
	$
	Thus, the proof is completed.

\end{proof}

\begin{example}
	
	Let $P = D_t + Q(t,D_x)$ be a system such that $Q(t,D_x) = c(t) [Q_{j,k}(D_x)]$.   
	In this case  $Q(t,\xi) = c(t)Q(\xi)$  is strongly triangularizable, since we can apply  Schur's triangularization for $Q(\xi)$. Hence,  $Pu = f$ is  equivalent to
	$$
	D_t v(t, \xi) + c(t)(\Lambda(\xi) + \mathcal{N}(\xi)) v(t, \xi) = g(t,\xi).
	$$

	The eigenvalues of this system are $\lambda_k(t, \xi) = c(t) \lambda_{k}(\xi)$ with averages $\lambda_{0,k}(\xi) = c_0 \lambda_{k}(\xi)$, where 
	$c_0 = (2\pi)^{-1}\int_{0}^{2\pi}c(t)dt.$ Then, we have 
	$$
	\mathscr{L}_k^{\xi} = D_t +  c(t) \lambda_k(\xi), \ \textrm{ and } \ \mathscr{L}_{0,k} = D_t +  c_0\lambda_k(D_x).
	$$

	Furthermore, if ($\mathscr{B}_4$) is fulfilled,  then $P$ is globally hypoelliptic if and only if there exist positive constants $C_k$, $M_k$ and $R_k$ such that
	$$
	|\tau+c_0 \lambda_{k}(\xi)|\geq C_k(|\tau| + |\xi|)^{-M_k}, \ \textrm{for all} \ |\tau| + |\xi| \geq R_k,
	$$
	for $k = 1, \ldots, m.$
\end{example}

\begin{example}
	Let $P = D_t + Q(t, D_x)$ be as defined in Example \eqref{exemple_ST} and assume that $a_0 = 0$. Thus, we  have 
	\begin{equation*}
		\mathscr{L}_1^{\xi} = D_t + a(t) + b(t)\xi \ \textrm{ and } \ 
		\mathscr{L}_2^{\xi} = D_t + a(t) - b(t)\xi.
	\end{equation*}
	
	Hence, system $P$ is globally hypoelliptic if and only if there exist constants satisfying
	\begin{equation*}
		|\tau \pm b_0 \xi|\geq C(|\tau| + |\xi|)^{-M}, \ \textrm{for all} \ |\tau| + |\xi| \geq R,
	\end{equation*}
	or equivalent, if  $b_0$ is  an irrational non-Liouville number.

\end{example}

\begin{example}
	
	Let $\{A_j\}_{j=1}^{n}$ be a family of $m\times m$ commuting matrices with eigenvalues $\{\sigma_{j,\ell}\}_{\ell=1}^{m}$. Given  $n$ pseudo-differential operators
	$P_j(D_{x_j})$, each one defined on $\mathbb{T}_{x_j}$, we can set the following system 
	\begin{equation*}
		P = D_t + \sum_{j=1}^{n} c_j(t) A_j P_j(D_{x_j}), \ (t,x) \in \mathbb{T}^{n+1}.
	\end{equation*}
	
	In view of the commutative assumption, we may apply the simultaneous triangularization to $\{A_j\}_{j=1}^{n}$. Then, the corresponding system  \eqref{general-system} can be rewritten as 
	$$
	D_tv(t,\xi) + \mathcal{A}(t, \xi)v(t,\xi) = g(t,\xi),
	$$
	where
	$$
	\mathcal{A}(t, \xi) =
	\left[
	\begin{array}{ccccc}
		\lambda_{1}(t, \xi) & \rho_{1,2}(t,\xi)  & \ldots & \rho_{1,m}(t,\xi) \\
		0 & \lambda_{2}(t, \xi)  & \ldots & \rho_{2,m}(t,\xi) \\
		\vdots & \vdots  & \vdots & \vdots \\
		0 & 0 & \ldots  & \lambda_{m}(t, \xi) 
	\end{array}
	\right] ,
	$$
	with
	$$
	\lambda_{\ell}(t, \xi) = \sum_{j=1}^{n}c_j(t)\sigma_{j,\ell} \, p_j(\xi_j), 
	\ \textrm{ and } \
	\rho_{\ell , k}(t,\xi) = \sum_{j=1}^{n}c_j(t)r_{\ell , k}^j p_j(\xi_j).
	$$
	
	In this case we have the operators 
	$$
	\mathscr{L}_{\ell} = D_t + \sum_{j=1}^{n}c_{j}(t) \sigma_{j,\ell} \, P_j(D_{x_j})
	\ \textrm{ and } \
	\mathscr{L}_{0,\ell} = D_t + \sum_{j=1}^{n}c_{0,j} \sigma_{j,\ell} \, P_j(D_{x_j}).
	$$
	Thus, if  $\lambda_{\ell}(t, \xi)$ satisfies ($\mathscr{B}_4$), then $P$ is globally hypoelliptic if and only if there are positive constants $C_\ell$, $M_\ell$ and $R_\ell$ such that
	$$
	\left|\tau+ \sum_{j=1}^{n}c_{0,j}\sigma_{j,\ell} \, p_j(\xi_j) \right|\geq C_\ell(|\tau| + |\xi|)^{-M_\ell}, \ \textrm{for all} \ |\tau| + |\xi| \geq R_\ell,
	$$
	for $\ell = 1, \ldots, m$.

\end{example}

%%%%%%%%%%%%%%%%%%%%%%%%%%%%%%%%%%%%%%%%%%%%%%%%%%%%%%%%%%%%%%%%
%%%%%%%%%%%%%%%%%%%%%%%%%%%%%%%%%%%%%%%%%%%%%%%%%%%%%%%%%%%%%%%%
\section{Reduction to the triangular form} \label{sec-4}
%%%%%%%%%%%%%%%%%%%%%%%%%%%%%%%%%%%%%%%%%%%%%%%%%%%%%%%%%%%%%%%%
%%%%%%%%%%%%%%%%%%%%%%%%%%%%%%%%%%%%%%%%%%%%%%%%%%%%%%%%%%%%%%%%

The aim of this Section is to present sufficient conditions on a matrix symbol $Q(t,\xi)$ such that it can be triangularized in the sense  of definition \eqref{def-stro-tri}.
The triangularization process shown in the sequel follow the same ideas used in \cite{Garetto2018}, hence
it is possible to omit some steps in this presentation. In view of this, attention will be focused on 
the verification that, under suitable hypotheses,  the triangularization
$$
S^{-1}(t, \xi) Q(t,\xi) S(t, \xi) =  \Lambda(t,\xi) + \mathcal{N}(t, \xi)
$$
is smooth, that is, the matrices $S$, $S^{-1}$ and $\mathcal{N}$ have entries belonging to $C^{\infty}(\mathbb{T} \times \mathbb{Z}^n)$.

Furthermore,  Theorem \ref{smooth-t-sym} exhibits  conditions such that 
$(\mathscr{B}_2)$ and $(\mathscr{B}_3)$ are fulfilled.

\begin{theorem}\label{schur-smooth-trian}
	Let $\mathcal{A}(t,\xi) = [a_{j,k}(t,\xi)]$ be a $m\times m$  matrix with 
	$a_{j,k}(t,\xi) \in C^{\infty}(\mathbb{T} \times \mathbb{Z}^n)$. Admit that its eigenvalues 	$\lambda_{1}(t,\xi),  \ldots, \lambda_{m}(t,\xi)$ and corresponding eigenvectors $h_{1}(t,\xi),  \ldots, h_{m-1}(t,\xi)$ are also smooth and that 
	$$
	\langle h^{(i)}(t, \xi) , e_1 \rangle \neq 0, \ \forall (t, \xi) \in \mathbb{T} \times \mathbb{Z}^n,
	$$
	holds for all $i = 1, \ldots, m-1$, with the notation to be explained, and $e_1=(1,0, \ldots,0)\in \mathbb{R}^{m-i+1}$.
	
	Then, there exist smooth matrices $S(t, \xi)$, invertible for all $(t,\xi) \in \mathbb{T} \times \mathbb{Z}^n$, with smooth inverses $S^{-1}(t, \xi)$, such that 
	$$
	S^{-1}(t, \xi) \mathcal{A}(t, \xi) S(t, \xi)  = \Lambda(t,\xi) + \mathcal{N}(t,\xi), 
	$$
	for all $(t, \xi) \in \mathbb{T} \times \mathbb{Z}^n$, where  $\Lambda(t, \xi) = diag (\lambda_{1}(t,\xi), \ldots, \lambda_{m}(t,\xi))$ and $\mathcal{N}(t,\xi)$ is an upper triangular nilpotent matrix with entries $r_{j,k}(t,\xi) \in C^{\infty}(\mathbb{T} \times \mathbb{Z}^n)$.
	
\end{theorem}

\begin{remark}
	The proof of this theorem is given as follows:  by Proposition \ref{schur-step} below,
	it is possible to start   by working with a  pair eigenvalue-eigenvector, say $\lambda-h$, and exhibit a decomposition 
	$$
	S^{-1} \mathcal{A} S =
	\left[
	\begin{array}{cc}
		\lambda & * \\
		0      & \mathcal{E}
	\end{array}
	\right], 
	$$
	for some matrix $\mathcal{E}$, of order $m-1$. Thus, the  procedure consists in  successive reapplications 
	of this process to  the matrix $\mathcal{E}$. 
\end{remark}

\begin{proposition}\label{schur-step}
	Let $\lambda(t,\xi)$ be an eigenvalue of $\mathcal{A}(t,\xi)$ and $h(t,\xi)$ an associated eigenvector with both of them smooth on $\mathbb{T} \times \mathbb{Z}^n$. Suppose that  there exists $j \in \{1, \ldots, m\}$ with
	\begin{equation}\label{hypo-eigenvectors}
		\langle h(t,\xi) , e_j \rangle_{\mathbb{C}^{m}} \neq 0, \ \forall (t,\xi) \in \mathbb{T} \times \mathbb{Z}^n,
	\end{equation}
	where  $e_j$ denotes the j-th element of the standard basis of $\mathbb{R}^m$.
	
	Then, there are a smooth  matrices $\mathcal{E}(t, \xi)$, of order $m-1$, and  $S(t, \xi)$ of order $m$, invertible for all $(t,\xi) \in \mathbb{T} \times \mathbb{Z}^n$, with a smooth  inverse $S^{-1}(t, \xi)$, such that 
	\begin{equation}\label{tring-step1}
		S^{-1}(t, \xi) \mathcal{A}(t,\xi) S(t, \xi) =
		\left[
		\begin{array}{ccccc}
			\lambda(t,\xi) & a_{1,2}(t,\xi)   & \ldots     &  a_{1,m}(t, \xi) \\
			0              &                     &                       &   \\
			\vdots         &                     & \mathcal{E}(t, \xi)   & \\
			0              &                     &                       &  
		\end{array}
		\right].
	\end{equation}

\end{proposition}

\begin{proof}
	Without loss of generality (see Remark 5 in \cite{Garetto2018}) we can admit $j=1$ in \eqref{hypo-eigenvectors}. Define
	$$
	v_i(t,\xi) = \dfrac{\langle h(t,\xi) , e_i \rangle_{\mathbb{C}^{m}}}{\langle h(t,\xi) , e_1 \rangle_{\mathbb{C}^{m}} } \in C^{\infty}(\mathbb{T} \times \mathbb{Z}^n), \ i= 1, \ldots, m,
	$$
	and the matrices
	$$
	S(t,\xi) = \left[
	\begin{array}{ccccc}
		v_1 & 0     & \ldots                &  0 \\
		v_2 &       &                       &   \\
		\vdots     &       & I_{m-1}   & \\
		v_m &       &                       &  
	\end{array}
	\right] 
	\ \textrm{ and } \
	S^{-1}(t,\xi) =	\left[
	\begin{array}{ccccc}
		v_1 & 0     & \ldots                &  0 \\
		-v_2&       &                       &    \\
		\vdots     &       & I_{m-1}               &    \\
		-	v_m&       &                       &  
	\end{array}
	\right], 
	$$
	where $I_{m-1}$ is the identity matrix of order $m-1$.
	
	With these constructions, it can be verified that \eqref{tring-step1} is fulfilled. Once each coordinate $v_i(t,\xi)$ is smooth, we obtain that  $S$, $S^{-1}$ and $\mathcal{E}$ are also  smooth on $\mathbb{T} \times \mathbb{Z}^n$.
	
\end{proof}

\subsection*{Proof of Theorem \ref{schur-smooth-trian}}

As observed, the full triangularization is given by applying  
Proposition  \ref{schur-step} to $\mathcal{E}(t, \xi)$ for $m-2$ times.   To present a sketch of this process, let $\mathcal{A}(t,\xi)$ be as in Theorem \ref{schur-smooth-trian} and $h_{1}(t,\xi),  \ldots, h_{m-1}(t,\xi)$ be the associated eigenvectors to the eigenvalues $\lambda_{1}(t,\xi),  \ldots, \lambda_{m-1}(t,\xi)$.

The construction of vectors $h^{(i)}(t, \xi)$  starts by setting $h^{(1)} \doteq h_1$ and, as in Proposition \ref{schur-step}, by assuming
$$
\langle h^{(1)}(t, \xi) , e_1 \rangle \neq 0, \ \forall (t, \xi) \in \mathbb{T} \times \mathbb{Z}^n.
$$

Hence, we obtain
\begin{equation*}
	S^{-1}_1  \mathcal{A}  S_1 =
	\left[
	\begin{array}{ccccc}
		\lambda_1 & a_{1,2}      & \ldots                &  a_{1,m} \\
		0              &                     &                       &   \\
		\vdots         &                     & \mathcal{E}_{m-1}   & \\
		0              &                     &                       &  
	\end{array}
	\right],
\end{equation*}
where $S_1 = [v_1 \ e_2 \ \ldots \ e_m]$, $v_1 =[v_{11} \ \ldots \ v_{1m}]^T$ and
$$
v_{1j}(t,\xi) = \dfrac{\langle h^{(1)}(t,\xi) , e_j \rangle_{\mathbb{C}^{m}}}{\langle h^{(1)}(t,\xi) , e_1 \rangle_{\mathbb{C}^{m}}}  \in C^{\infty}(\mathbb{T} \times \mathbb{Z}^n).
$$

Now, we shall see  how to apply  this process to $\mathcal{E}_{m-1}$. For  this,  
consider the projector $\pi_k: \mathbb{R}^m \rightarrow \mathbb{R}^{m-k}$, given by
$\pi_k (x_1, \ldots, x_m) = (x_{k+1}, \ldots, x_m).$

Notice that $S_1^{-1}h_2$ is an eigenvector of $S_1^{-1}AS_1$, with eigenvalue $\lambda_2$. Furthermore,  
$$
h^{(2)} \doteq ( \pi_1 \circ S_1^{-1} ) \cdot h_2 
$$
is an eigenvector of $\mathcal{E}_{m-1}$, smooth on $\mathbb{T} \times \mathbb{Z}^n$, and  corresponding to $\lambda_2$.

By assuming 
$$
\langle h^{(2)}(t, \xi) , e_1 \rangle \neq 0, \ \forall (t, \xi) \in \mathbb{T} \times \mathbb{Z}^n,
$$
for $e_1 = (1, 0, \ldots, 0) \in \mathbb{R}^{m-1}$, we get 
the following decomposition from Proposition  \ref{schur-step} 

\begin{equation*}
	\widetilde{S}^{-1}_2  \mathcal{E}_{m-1}  \widetilde{S}_2 =
	\left[
	\begin{array}{ccccc}
		\lambda_2      & *      & \ldots                &  * \\
		0              &                     &                       &   \\
		\vdots         &                     & \mathcal{E}_{m-2}   & \\
		0              &                     &                       &  
	\end{array}
	\right],
\end{equation*}
where the first row is the first row of $\mathcal{E}_{m-1}$. The matrix $\widetilde{S}_2 = [v_{2} \ e_2 \ \ldots \ e_{m-1} ]$ is computed by setting
$v_2 = [v_{22} \ \ldots \ v_{2m}]^T$ and 
$$
v_{2j}(t,\xi) = \dfrac{\langle h^{(2)}(t,\xi) , e_j \rangle_{\mathbb{C}^{m-1}}}{\langle h^{(2)}(t,\xi) , e_1 \rangle_{\mathbb{C}^{m-1}}} \in C^{\infty}(\mathbb{T} \times \mathbb{Z}^n).
$$

Moreover, by defining 
\begin{equation*}
	S_2 =
	\left[
	\begin{array}{ccccc}
		1      & 0      & \ldots                &  0 \\
		0              &                     &                       &   \\
		\vdots         &                     & \widetilde{S}_2   & \\
		0              &                     &                       &  
	\end{array}
	\right] 
\end{equation*}
we get
\begin{equation*}
	S^{-1}_2S^{-1}_1  \mathcal{A}  S_1S_2 =
	\left[
	\begin{array}{cccccc}
		\lambda_1      & *           & * & \ldots                &  * \\
		0              &  \lambda_2  & * &  \ldots                 &  * \\
		0              &  0          & &                       &   \\
		\vdots         &   \vdots    & & \mathcal{E}_{m-2}     & \\
		0              &   0         &  &                       &  
	\end{array}
	\right]. 
\end{equation*}

The full triangularization is  obtained  by using an inductive argument on the order of the matrices $\mathcal{E}_{j}$. The exhibition is restricted to show only how to construct vectors $h^{(i)}$ and   matrices $S_k$. For $i \in \{2, \ldots, m-1\}$, let us set 
\begin{equation}\label{const_h_i}
	h^{(i)} \doteq \left( \pi_{i-1} \circ S_{i-1}^{-1} \circ S_{i-2}^{-1} \circ \ldots \circ S_{1}^{-1}   \right) \cdot h_{i},
\end{equation}
which are evidently smooth on $\mathbb{T} \times \mathbb{Z}^n$. The matrices $S_k$ are inductively  defined by
\begin{equation}\label{Sk}
	S_k =
	\left[
	\begin{array}{cc}
		I_{k-1} & 0 \\
		0     & \widetilde{S}_{k}
	\end{array}
	\right], 
	\ \textrm{ with } \
	\widetilde{S}_{k} =[v_k \ e_2 \ \ldots \ e_{m-k}],
\end{equation}
where $e_i \in \mathbb{R}^{m-k}$ and 
\begin{equation}\label{vk}
	v_{kj}(t,\xi) = \dfrac{\langle h^{(k)}(t,\xi) , e_j \rangle_{\mathbb{C}^{m-k}}}{\langle h^{(k)}(t,\xi) , e_1 \rangle_{\mathbb{C}^{m-k}}}  \in C^{\infty}(\mathbb{T} \times \mathbb{Z}^n), \ j =k, \ldots, m.
\end{equation}

Finally,   $S(t,\xi)$ is the smooth matrix given by 
\begin{equation}\label{S}
	S = S_1 \circ S_2 \circ \ldots \circ S_{m-1}.
\end{equation}

\qed

\subsection{Strongly triangularizable symbols}

Notice that a matrix symbol $Q(t,\xi)$ that satisfies both condition ($\mathscr{A}$) and the 
hypotheses of Theorem \ref{schur-smooth-trian}  can be smoothly  triangularized. Hence, conditions such that $(\mathscr{B}_2)$ and  $(\mathscr{B}_3)$  are fulfilled, are all that remains for the strong triangularization of  $Q(t,\xi)$. This is the aim of the next Theorem.

\begin{theorem}\label{smooth-t-sym}
	Let  $Q(t,\xi)$ be a symbol satisfying condition  ($\mathscr{A}$) and the hypotheses of Theorem \ref{schur-smooth-trian}. Assume that the eigenvectors  $h_\ell(t, \xi)$, $\ell =1, \ldots, m-1$, satisfy the following properties:
	\begin{enumerate}
		\item [(a)] there are  constants $\kappa,\delta\geq0$ and $C_1,C_2,R>0$   satisfying
		\begin{equation}\label{bound_zero_order}
			C_1|\xi|^{\kappa} \leq \sup_{t \in \mathbb{T}}|h_{\ell}(t,\xi)| \leq C_2  |\xi|^{\delta}, \  |\xi|\geq R;
		\end{equation}

		\item [(b)] given $p \in \mathbb{N}$ and $N>0$ there are positive constants $C$ and $R$ such that 
		\begin{equation}\label{rap_dec_eigne}
			\sup_{t \in \mathbb{T}}	|\partial_t ^{p} h_\ell(t, \xi)| \leq C|\xi|^{-N}, \ \forall |\xi|\geq R.
		\end{equation}	
		
	\end{enumerate}

	Under these conditions,   $Q(t, \xi)$ is strongly triangularizable.

\end{theorem}

\begin{proof}
	The starting point is to prove that $(\mathscr{B}_2)$ is fulfilled. In  view of the construction given  in the proof of Theorem \ref{schur-smooth-trian} (especially equations \eqref{Sk} and \eqref{S}),  it is enough to show that each matrix $\widetilde{S}_k$  satisfies \eqref{poly-decay-S-sec4}. Hence, we should analyze the derivatives of the entries $v_{kj}(t,\xi)$ given by \eqref{vk}.  
	
	To simplify the notations let us  proceed by omitting the index $\ell$ in the eigenvectors. Thus, writing  $h^{(k)}(t,\xi) = (h^{(k)}_{1}(t,\xi), \ldots, h^{(k)}_{m-k}(t,\xi))$ we obtain 
	$$
	v_{kj}(t,\xi) = \dfrac{\langle h^{(k)}(t,\xi) , e_j \rangle_{\mathbb{C}^{m-k}}}{\langle h^{(k)}(t,\xi) , e_1 \rangle_{\mathbb{C}^{m-k}}} =
	\dfrac{h^{(k)}_j(t,\xi)}{h^{(k)}_1(t,\xi)}.
	$$

	In particular, it follows from \eqref{const_h_i} and condition 
	(a)  the existence of constants $C_3$ and $\eta_1$ such that $|v_{kj}(t,\xi)| \leq C_3|\xi|^{\eta_1}$. Moreover, by  \eqref{S} we obtain new constants satisfying
	\begin{equation}\label{order_zero}
		\sup_{t \in \mathbb{T}}\|S(t, \xi)\|_{\mathbb{C}^{m \times m}} \leq C|\xi|^{\eta} \ \textrm{ and } \ 
		\sup_{t \in \mathbb{T}}\|S^{-1}(t, \xi)\|_{\mathbb{C}^{m \times m}}  \leq C|\xi|^{\eta}.
	\end{equation}

	Now, in order to investigate the derivatives of $v_{kj}$ we observe that  by the Leibniz formula we get
	$$
	\partial^{\alpha}_t  v_{kj} = \sum_{\gamma = 0}^{\alpha}\binom{\alpha}{\gamma}\partial_t^{\alpha-\gamma}  h^{(k)}_j \, \partial_t^{\gamma}  \left( \dfrac{1}{ h^{(k)}_1} \right), \ \alpha \in \mathbb{N}.
	$$
	
	Once more,  to simplify notation, write  $\omega = h^{(k)}_1$. It follows from the  Fa\`a di Bruno's formula that
	$$
	\partial_t^{\gamma}  \left( \omega^{-1} \right) =\sum_{\Delta(\gamma)}
	\left[
	\dfrac{\gamma !}{\beta!} (-1)^{|\beta|} |\beta|! \left( \omega^{-1} \right)^{|\beta| +1} \prod_{s=1}^{\gamma}\left(\dfrac{\partial_t^{s}\omega}{s !} \right)^{\beta_s}\right],
	$$
	where $\Delta(\gamma) = \{ \beta \in \mathbb{Z}^{\gamma}_+; \ \sum_{s=1}^{\gamma}s\,\beta_s = \gamma \}$.
	
	Consider $N>0$. By equations \eqref{const_h_i}, \eqref{bound_zero_order} and \eqref{rap_dec_eigne} we get 
	$$
	\left|\dfrac{\partial_t^{s}\omega}{s !} \right|^{\beta_s} \leq C_4 |\xi|^{-N\beta_s} (s!)^{-\beta_s}  \leq C_4 (s!)^{-\beta_s},
	$$
	for $|\xi|$ large enough. Also, there exist $\eta_{2}\geq 0$ such that $|\omega |^{-1} \leq C_5|\xi|^{-\eta_2}$. Hence, 
	$$
	|\omega |^{-(|\beta| +1)}\leq C_5 |\xi|^{-\eta_2 (|\beta| +1)} \leq C_5, \ \textrm{ as } \ |\xi| \to \infty,
	$$
	and we may obtain $R_2>0$ such that
	\begin{equation}\label{bound-proof}
		|\partial_t^{\gamma}  \left( \omega^{-1} \right)|  \leq  \sum_{\Delta(\gamma)}
		\left[
		\dfrac{\gamma !}{\beta!} \, \, |\beta|! \, \,    |\omega|^{-(|\beta| +1)}
		\prod_{s=1}^{\gamma}\left|\dfrac{\partial_t^{s}\omega}{s !} \right|^{\beta_s}\right]  
		\leq \sum_{\Delta(\gamma)}
		\left[
		\dfrac{\gamma !}{\beta!} \, \, |\beta|! \, C_5
		\prod_{s=1}^{\gamma}C_4(s!)^{-\beta_s}\right]   
		= C_6,	
	\end{equation}
	for all $|\xi| \geq R_2$. 
	
	Since $|\partial_t^{\alpha-\gamma}  h^{(k)}_j(t,\xi)| \leq C_7|\xi|^{-N}$, it follows from \eqref{bound-proof} that
	\begin{equation}\label{any_order}
		|\partial^{\alpha}_t  v_{kj}(t,\xi)|  \leq \sum_{\gamma = 0}^{\alpha}\binom{\alpha}{\gamma} |\partial_t^{\alpha-\gamma}  h^{(k)}_j(t,\xi)| \left|\partial_t^{\gamma}  \left( \omega^{-1} \right)\right| \\
		\leq C |\xi|^{-N}, \ \textrm{ as } \ |\xi| \to \infty.
	\end{equation}
	Hence,  $(\mathscr{B}_2)$ is a consequence of \eqref{S}, \eqref{order_zero} and \eqref{any_order}.
	
	Finally, note that 
	$$
	D_t S(t,\xi) = 
	\left[
	\begin{array}{ccccc}
		0 & 0     & \ldots   &             &  0 \\
		D_tv_{2,1}&  0     &   &                   & \vdots   \\
		\vdots     &    \ddots   &  &           &    \\
		&     & &          &  	\\
		D_tv_{m,1}&   \ldots      &      &              D_tv_{m,m-1}   & 0 
	\end{array}
	\right].
	$$
	Then, given	$\beta \in \mathbb{Z}_+$ and $N>0$ we obtain $C,R>0$ such that
	\begin{equation}\label{estimate_D}
		\sup_{t \in \mathbb{T}}\|\partial^{\beta}_tD_t S(t,\xi) \|_{\mathbb{C}^{m \times m}} 
		\leq C |\xi|^{-N}, \ |\xi|\geq R.
	\end{equation}

	Therefore, 	$(\mathscr{B}_3)$ is a consequence of $(\mathscr{B}_2)$,  \eqref{estimate_D} and
	$$
	\|\partial^{\alpha}_tB(t,\xi)\|_{\mathbb{C}^{m \times m}}  \leq 
	\sum_{\beta = 0}^{\alpha}  \binom{\alpha}{\beta}\|\partial^{\alpha-\beta}_tS^{-1}(t,\xi) \|_{\mathbb{C}^{m \times m}} \|\partial^{\beta}_tD_t S(t,\xi) \|_{\mathbb{C}^{m \times m}}.	
	$$

\end{proof}

\begin{example}
	Consider
	$$
	Q (t,D_x) = \left [
	\begin{array}{cc}
		0  &  a^2(t) D_x^2\\[2mm]
		P(D_x)  & 0
	\end{array} \right], \ (t,x) \in  \mathbb{T}^2,
	$$
	where $a(t)\not \neq 0$ is a smooth real-valued function and $P(D_x)$ is a pseudo-differential operator on $\mathbb{T}$ with symbol $p(\xi)$ satisfying the following condition: for all $N>0$ there exists $C,R>0$ such that
	\begin{equation}\label{rep_decay_symbol}
		0\leq p(\xi) \leq C |\xi|^{-N}, \ \forall |\xi| \geq R.
	\end{equation}
	
	The eigenvalues are $\pm a(t)\xi\sqrt{p(\xi)}$ with corresponding eigenvectors
	$$
	h(t,\xi) =
	\left\{
	\begin{array}{l}
		(1, \pm 1), \ \xi =0, \\
		\left( 1, \pm \dfrac{\sqrt{p(\xi)}}{\xi a(t)}\right), \ \xi \neq 0.
	\end{array}
	\right.
	$$

	It follows from \eqref{rep_decay_symbol} that conditions (a) and (b) in Theorem \ref{smooth-t-sym} are fulfilled. In particular, 
	$$
	\left [
	\begin{array}{cc}
		1  &  0\\[2mm]
		-\dfrac{\sqrt{p(\xi)}}{\xi a(t)} & 1
	\end{array} \right]
	\, 
	Q (t,\xi)
	\,  \left [
	\begin{array}{cc}
		1  &  0\\[2mm]
		\dfrac{\sqrt{p(\xi)}}{\xi a(t)}  & 1
	\end{array} \right]
	=
	\left [
	\begin{array}{cc}
		a(t)\xi\sqrt{p(\xi)}  &   a^2(t) \xi^2\\[2mm]
		0  & - a(t)\xi\sqrt{p(\xi)} \xi
	\end{array} \right]
	$$
	and 
	$$
	D_tB(t,\xi) =  \dfrac{a'(t) \sqrt{p(\xi)}}{a^2(t) \xi}	 
	\left [
	\begin{array}{cc}
		0 &   0\\[2mm]
		i & 0 
	\end{array} \right], \ \xi \neq 0.
	$$

\end{example}

\bibliography{deAvilaSilva_Fernando}

\begin{thebibliography}{10}
\providecommand{\url}[1]{\texttt{#1}}
\providecommand{\urlprefix}{URL }
\expandafter\ifx\csname urlstyle\endcsname\relax
  \providecommand{\doi}[1]{doi:\discretionary{}{}{}#1}\else
  \providecommand{\doi}{doi:\discretionary{}{}{}\begingroup
  \urlstyle{rm}\Url}\fi

\bibitem{Araujo2019}
G.~{Araújo}, \textit{Global regularity and solvability of left-invariant
  differential systems on compact lie groups}, Ann. Global Anal. Geom.
  \textbf{56} (2019), no.~4, 631--665.
  \urlprefix\url{https://doi.org/10.1007/s10455-019-09682-9}.

\bibitem{BCM}
A.~P. {Bergamasco}, P.~D.~{Cordaro}, and P.~{Malagutti}, \textit{Globally
  hypoelliptic systems of vector fields}, J. Funct. Anal.
  \textbf{114} (1993), no.~2, 267 -- 285.
  \urlprefix\url{https://doi.org/10.1006/jfan.1993.1068}.

\bibitem{BCP04}
A.~P. {Bergamasco}, P.~D.~{Cordaro}, and G.~{Petronilho}, \textit{Global
  solvability for a class of complex vector fields on the two-torus}, Comm.
  Partial Differential Equations \textbf{29} (2004), no. 5-6, 785--819.
  \urlprefix\url{https://www.tandfonline.com/doi/abs/10.1081/PDE-120037332}.

\bibitem{BK}
A.~P. {Bergamasco} and A.~{Kirilov}, \textit{Global solvability for a class of
  overdetermined systems}, J. Funct. Anal. \textbf{252} (2007),
  no.~2, 603--629.
  \urlprefix\url{https://www.sciencedirect.com/science/article/pii/S0022123607000973}.

\bibitem{BMZ}
A.~P. {Bergamasco}, C.~{Medeira}, and S.~{Zani}, \textit{Globally solvable
  systems of complex vector fields}, J. Differential Equations
  \textbf{252} (2012), no.~8, 4598--4623.
  \urlprefix\url{https://www.sciencedirect.com/science/article/pii/S0022039612000253}.


\bibitem{BKWS}
A.~P. {Bergamasco}, A.~{Kirilov}, W.~V.~L.~{Nunes} and S.~{Zani}, \textit{On the global solvability for overdetermined
  systems}, Trans. Amer. Math. Soc. \textbf{364} (2012), no.~9, 4533--4549.
  \urlprefix\url{http://www.jstor.org/stable/41639043}.


\bibitem{BKNZ15}
A.~P. {Bergamasco}, A.~{Kirilov}, W.~V.~L.~{Nunes} and S.~{Zani}, \textit{Global solutions to involutive systems.},
  Proc. Amer. Math. Soc. \textbf{143} (2015), no.~11, 4851--4862.
  \urlprefix\url{https://www.ams.org/journals/proc/2015-143-11/S0002-9939-2015-12633-0/}.


\bibitem{BPZaZug17}
A.~P. {Bergamasco}, A.~{Parmeggiani}, S.~{Zani} and  G.~{Zugliani}, \textit{Classes of globally solvable involutive
  systems.}, J. Pseudo-Differ. Oper. Appl. \textbf{8} (2017), no.~4, 551--583.
  \urlprefix\url{https://link.springer.com/article/10.1007/s11868-017-0217-9}.


\bibitem{BPZZ}
A.~P. {Bergamasco}, A.~{Parmeggiani}, S.~{Zani} and  G.~{Zugliani}, \textit{Geometrical proofs for the global
  solvability of systems}, Math. Nachr. \textbf{291} (2018),
  no.~16, 2367--2380.
  \urlprefix\url{https://onlinelibrary.wiley.com/doi/abs/10.1002/mana.201700300}.

\bibitem{C-CHINI}
G.~{Chinni} and P.~D.~{Cordaro}, \textit{On global analytic and {G}evrey
  hypoellipticity on the torus and the {M}{\'e}tivier inequality}, Comm.
  Partial Differential Equations \textbf{42} (2017), no.~1, 121--141.
  \urlprefix\url{https://www.tandfonline.com/doi/abs/10.1080/03605302.2016.1258577?journalCode=lpde20}.

\bibitem{Avila}
F.~{de {\'A}vila Silva}, \textit{Global hypoellipticity for a class of periodic
  cauchy operators}, J. Math. Anal. Appl. \textbf{483} (2020), no.~2, 123650.
  \urlprefix\url{https://doi.org/10.1016/j.jmaa.2019.123650}.

\bibitem{AK19}
F.~{de {\'A}vila Silva} and A.~{Kirilov}, \textit{Perturbations of globally
  hypoelliptic operators on closed manifolds}, J. Spectr. Theory \textbf{9}
  (2019), no.~3, 825--855.
  \urlprefix\url{https://ems.press/journals/jst/articles/15811}.

\bibitem{AvilaMedeira}
F.~{de {\'A}vila Silva} and C.~{Medeira}, \textit{Global hypoellipticity for a
  class of overdetermined systems of pseudo-differential operators on the
  torus}, Ann. Mat. Pura Appl. \textbf{200} (2021), no.~6, 2535--2560.
  \urlprefix\url{https://doi.org/10.1007/s10231-021-01090-w}.



\bibitem{AGKM}
F.~{de {\'A}vila Silva}, R.~B.~{Gonzalez}, A.~{Kirilov} and C.~{Medeira}, \textit{Global hypoellipticity for a class of
  pseudo-differential operators on the torus}, J. Fourier Anal. Appl.
  \textbf{25} (2019), no.~4, 1717--1758.
  \urlprefix\url{https://link.springer.com/article/10.1007/s00041-018-09645-x}.

\bibitem{DGY02}
D.~{Dickinson}, T.~V. {Gramchev}, and M.~{Yoshino}, \textit{Perturbations of
  vector fields on tori: resonant normal forms and {D}iophantine phenomena},
  Proc. Edinb. Math. Soc. \textbf{45} (2002), 731--759.
  \urlprefix\url{https://doi.org/10.1017/S001309150000064X}.

\bibitem{Garetto2018}
C.~{Garetto}, C.~{J{\"a}h}, and M.~{Ruzhansky}, \textit{Hyperbolic systems with
  non-diagonalisable principal part and variable multiplicities, i:
  well-posedness}, Math. Ann. \textbf{372} (2018), no.~3,
  1597--1629.
  \urlprefix\url{https://link.springer.com/article/10.1007/s00208-018-1672-1}.

\bibitem{Gramchev2013}
T.~V. {Gramchev} and M.~{Ruzhansky}, \textit{Cauchy Problem for Some {$2 \times
  2$} Hyperbolic Systems of Pseudo-differential Equations with
  Nondiagonalisable Principal Part}, Studies in Phase Space Analysis with
  Applications to PDEs. Progress in Nonlinear Differential Equations and Their
  Applications, {\textbf{84}}, (2013). 129--145,
  \doi{10.1007/978-1-4614-6348-1_7}.
  \urlprefix\url{https://link.springer.com/chapter/10.1007/978-1-4614-6348-1_7}.

\bibitem{GW1}
S.~J. {Greenfield} and N.~R. {Wallach}, \textit{Global hypoellipticity and
  {L}iouville numbers}, Proc. Amer. Math. Soc. \textbf{31} (1972), no.~1,
  112--114. \urlprefix\url{https://doi.org/10.2307/2038523}.

\bibitem{HIMONASGER}
A.~A. {Himonas} and G.~{Petronilho}, \textit{Global hypoellipticity and
  simultaneous approximability}, J. Funct. Anal. \textbf{170} (2000), no.~2,
  356--365. \urlprefix\url{https://doi.org/10.1006/jfan.1999.3524}.

\bibitem{Hof}
K.~{Hoffman} and R.~{Kunze}, \textit{{Linear algebra}}, 1st edn.,
  Prentice-Hall, 1971\doi{doi:10.2307/3617032}.

\bibitem{HOU79}
J.~{Hounie}, \textit{Globally hypoelliptic and globally solvable first-order
  evolution equations}, Trans. Amer. Math. Soc. \textbf{252} (1979), no.~12,
  233--248. \urlprefix\url{https://www.jstor.org/stable/1998087}.

\bibitem{HZ}
J.~{Hounie} and G.~{Zugliani}, \textit{Global solvability of real analytic
  involutive systems on compact manifolds. part 2}, Trans. Amer. Math. Soc.
  \textbf{371} (2019), no.~7, 5157--5178.
  \urlprefix\url{https://www.ams.org/journals/tran/2019-371-07/S0002-9947-2018-07718-2/home.html}.

\bibitem{Jamison54}
S.~L. {Jamison}, \textit{Perturbation of normal operators}, Proc. Amer. Math.
  Soc. \textbf{5} (1954), 103--110.
  \urlprefix\url{https://www.ams.org/journals/proc/1954-005-01/S0002-9939-1954-0060743-0/}.

\bibitem{kato}
T.~{Kato}, \textit{{Perturbation theory for linear operators.}}, reprint of the
  corr. print. of the 2nd ed. 1980 edn., Berlin: Springer-Verlag, 1995,
  \doi{10.1007/978-3-642-66282-9}.
  \urlprefix\url{https://www.springer.com/gp/book/9783540586616}.

\bibitem{KMR}
A.~{Kirilov}, W.~A.~A. {Moraes}, and M.~{Ruzhansky}, \textit{Global
  hypoellipticity and global solvability for vector fields on compact lie
  groups}, J. Funct. Anal. \textbf{280} (2021), no.~2, 108806.
  \urlprefix\url{https://www.sciencedirect.com/science/article/pii/S0022123620303499}.

\bibitem{Petr11}
G.~{Petronilho}, \textit{Global hypoellipticity, global solvability and normal
  form for a class of real vector fields on a torus and application}, Trans.
  Amer. Math. Soc. \textbf{363} (2011), no.~12, 6337--6349.
  \urlprefix\url{https://www.ams.org/journals/tran/2011-363-12/S0002-9947-2011-05359-6/}.

\bibitem{Rellich}
F.~{Rellich}, \textit{{Perturbation theory of eigenvalue problems.}}, {Notes on
  Mathematics and its Applications. New York-London-Paris: Gordon and Breach
  Science Publishers. X, 127 p. }, 1969.

\bibitem{BCCJ16}
N.~{Rodrigues}, G.~{Chinni}, P.~D.~{Cordaro} and M.~R.~{Jahnke}, \textit{{Lower order perturbation and global analytic
  vectors for a class of globally analytic hypoelliptic operators.}}, {Proc.
  Amer. Math. Soc.} \textbf{144} (2016), no.~12, 5159--5170.
  \urlprefix\url{https://www.ams.org/journals/proc/2016-144-12/S0002-9939-2016-13178-X/}.

\bibitem{RT3}
M.~{Ruzhansky} and V.~{Turunen}, \textit{Pseudo-Differential Operators and
  Symmetries: Background Analysis and Advanced Topics, Series:
  Pseudo-Differential Operators}, 2nd edn., Birkhäuser Basel, 2010,
  \doi{10.1007/978-3-7643-8514-9}.
  \urlprefix\url{https://link.springer.com/book/10.1007/978-3-7643-8514-9}.

\end{thebibliography}

\end{document}